\documentclass[11pt,a4paper]{amsart}

\usepackage[utf8]{inputenc}
\usepackage{txfonts}
\usepackage{hyperref}
\usepackage{graphicx}
\usepackage{amssymb}
\usepackage{amsmath}

\usepackage{tikz}
\usepackage{tikz-cd}
\usetikzlibrary{arrows}
\usepackage{mathrsfs}
\usepackage{mathdots}

\usepackage{marginnote}
\usepackage{todonotes}

\newcommand{\C}{\mathbb{C}}

\newcommand{\Z}{\mathbb{Z}}

\newcommand{\GL}{\mathbf{GL}}
\newcommand{\SL}{\mathbf{SL}}

\newcommand{\Hom}{\mathrm{Hom}}

\newcommand{\Aut}{\mathrm{Aut}}
\newcommand{\Ad}{\mathrm{Ad}}
\newcommand{\End}{\mathrm{End}}
\newcommand{\Lie}{\mathrm{Lie}}
\newcommand{\Maps}{\mathrm{Maps}}
\newcommand{\Fr}{\mathrm{Fr}}
\newcommand{\Isom}{\mathrm{Isom}}

\newcommand{\si}{\sigma}
\newcommand{\eps}{\varepsilon}
\newcommand{\lra}{\longrightarrow}
\newcommand{\lmt}{\longmapsto}
\newcommand{\sG}{\mathscr{G}}
\newcommand{\sH}{\mathscr{H}}
\newcommand{\sO}{\mathscr{O}}
\newcommand{\sP}{\mathscr{P}}
\newcommand{\hP}{\widehat{\mathscr{P}}}
\newcommand{\hG}{\widehat{\mathscr{G}}}

\newcommand{\sQ}{\mathscr{Q}}
\newcommand{\piX}{\pi_1 X}
\newcommand{\Xt}{\widetilde{X}}
\newcommand{\heta}{\widehat{\eta}}
\newcommand{\hzeta}{\widehat{\zeta}}

\renewcommand{\rho}{\varrho}
\renewcommand{\phi}{\varphi}

\renewcommand{\geq}{\geqslant}

\renewcommand{\phi}{\varphi}

\newtheorem{theorem}{Theorem}[section]
\newtheorem{proposition}[theorem]{Proposition}

\newtheorem{corollary}[theorem]{Corollary}

\theoremstyle{definition}
\newtheorem{definition}[theorem]{Definition}
\newtheorem{remark}[theorem]{Remark}
\newtheorem{example}[theorem]{Example}
\newtheorem{exercise}[theorem]{Exercise}

\numberwithin{equation}{section}
\numberwithin{figure}{section}

\begin{document}

\title{Flat torsors in complex geometry}

\author{Juan Sebastián Numpaque-Roa}
\address{Centro de Matemática da Universidade do Porto, Departamento de Matemática, Faculdade de Ciências da Universidade do Porto, Rua do Campo Alegre S/N, 4169-007 Porto, Portugal}
\email{js.numpaqueroa@gmail.com}

\author{Florent Schaffhauser}
\address{Institut für Mathematik, Universität Heidelberg, Im Neuenheimer Feld 205, 
69120 Heidelberg, Germany}
\email{fschaffhauser@mathi.uni-heidelberg.de}


\begin{abstract}
These notes grew out of a mini-course given by the second-named author at \textit{Casa Matemática Oaxaca} in the Fall of 2022. Their purpose is to provide an exposition, directed at graduate students, of the basic properties of complex analytic group bundles and torsors under them, including the flat case.
\end{abstract}

\maketitle

\tableofcontents
\section{Introduction}

Local systems of vector spaces appear naturally in the theory of ordinary differential equations: if $U \subset \C$ is an open set and $A : U \to \mathrm{Mat}(n\times n, \C)$ is a holomorphic map, the sheaf whose stalk at a point $z \in U$ consists of germs of solutions of a vector-valued linear differential equation $v'(z) = A(z) v(z)$ is a locally constant sheaf of finite dimensional vector spaces. Indeed, the germ at a point $z_0$ of a solution $v : U \to \C^n$ is entirely determined by the initial condition $v(z_0) \in \C^n$. Equivalently, we can view this local system of vector spaces as a vector bundle with \emph{locally constant transition functions}. Or yet equivalently, as a principal $\GL(n,\C)$-bundle with locally constant transition functions. In these notes, we study a generalization of this that enables us to take into account the symmetries of a differential equation. One way to encode such symmetries is to replace principal bundles by so-called \emph{torsors under non-constant groups}, where by non-constant group we mean a group bundle $\sG$ over a space $X$ (Definition \ref{defGrpBdle}). And by torsor we mean a fibre bundle $\sP$ over $X$, endowed with a fibrewise simply transitive action of the group bundle $\sG$ (Definition \ref{torsor_def}). Our goal is to introduce these notions in a certain amount of detail, assuming only basic familiarity with the theory of fibre bundles (\cite{Steenrod, Grothendieck_Kansas}). We shall see that, under certain assumptions, it is possible to classify all group bundles that are trivialized by a given Galois cover of $X$. This is the content of Theorem \ref{ClassifGpBdleTrivByCover}, which is the main result of Section \ref{GpSpace}. It implies that, if the universal cover of $X$ is a contractible Stein manifold $\Xt$ then, given a complex Lie group $G$, we can classify \emph{all} holomorphic group bundles with typical fibre $G$ on $X$ in terms of crossed morphisms $\phi : \piX \to \Maps(\Xt,\Aut(G))$.

\smallskip

Next, we study torsors under non-constant group bundles, which constitute a generalization of principal bundles. Essentially, a torsor is a fibre space $\sP$ over $X$, endowed with a fibrewise simply transitive action of a group bundle (Definition \ref{torsor_def}). This means that every fibre of a torsor over is still a principal homogeneous space under a certain group, but that group may vary depending on the fibre we look at. Remarkably, it is still possible in certain cases to classify torsors under a given group bundle $\sG$, either using the sheaf of local sections of $\sG$ (Theorem \ref{isomorphismH1andIsoClassesOfTorsors}) or, under the same assumptions on the $X$ as above, via crossed morphisms $\psi : \piX \to \Maps(\Xt, G)$. This last part is the content of Corollary \ref{pi_G_ppal_bundles}, which is obtained as a special case of a more general, but more abstract, classification result (Theorem \ref{ClassifTorsorTrivByCover}). Following \cite{Grothendieck_Kansas}, we also introduce in that section the general machinery of frame bundles (Theorem \ref{Aut_E_torsors}).

\smallskip

In the final section, we classify group objects and torsors under them in several categories: topological coverings of a connected topological space $X$ (assuming that $X$ admits a universal covering), analytic coverings of a connected complex analytic manifold, flat fibre bundles over a connected topological space, and flat fibre bundles over a connected complex analytic manifold. In Theorems \ref{ClassifGpCovering} and \ref{pi_Gamma_ppal_coverings}, for instance, we classify group coverings on a topological space $X$ and torsors under them, assuming only that $X$ admits a universal covering. The statements and proofs of these results remain valid in the analytic category: if we assume that $X$ is a complex analytic manifold admitting a universal covering, then we can classify analytic group coverings over $X$ and torsors under them. In Section \ref{twisted_local_systems}, we introduce and classify \emph{twisted local systems} on a complex analytic manifold $X$. These twisted local systems (Definition \ref{def_twisted_local_system}) are the objects that can encode differential equations with hidden symmetry and, by seeing them as torsors under a certain flat group bundle, we classify twisted local systems via crossed morphisms $\rho : \piX \to G$ (Theorem \ref{classif_twisted_local_systems}).

\medskip

\textbf{Acknowledgements.}
The first-named author is supported by FCT-Fundação para a Ciência e a Tecnologia, I.P., under the grant UI/BD/154369/2023 and partially supported by CMUP under the projects with 
reference UIDB/00144/2020 and UIDP/00144/2020 all funded by FCT with national funds. This project was partly supported by the Heidelberg University Excellence Initiative Mobility Grant ExU 11.2.1.45 and the AEI-DFG Joint Research Program V-SHaRP SCHA 2147\_1-1 AOBJ. The authors wish to thank the referee for their helpful comments and suggestions.

\section{Group bundles}\label{GpSpace}

Let $X$ be a topological space. Intuitively, a \textit{group object} in the category of spaces over $X$, or \textit{group space}, is a \textit{family of groups parameterised by} $X$. This means a quadruple $(p, \mu, \iota, \varepsilon)$ where
\begin{itemize}
    \item $p : \sG \to X$ is a space over $X$,
    \item $\mu : \sG\times_X\sG \to \sG$ is a morphism of spaces over $X$,
    \item $\iota : \sG \to \sG$ is a morphism of spaces over $X$, and
    \item $\varepsilon : X \to \sG$ is a section of $p$,
\end{itemize} 
such that for all $x \in X$, the restriction of $\mu$ and $\iota$ to the fibres of $p$ above $x$ defines a group $(\sG_x, \mu_x, \iota_x, \varepsilon(x))$, with $\mu_x$ as multiplication, $\iota_x$ as inversion, and $\varepsilon(x)$ as neutral element of the group $\sG_x$. 

\smallskip

Our endgoal in these notes is to work over a complex analytic \textit{manifold} $X$. So to make the definition above work in that settting, we need to guarantee that the fibre product $\sG \times_X \sG$ exists as a complex analytic manifold. To that end, it suffices to require that $p$ be a \emph{submersion}. In this section, we will restrict our attention to a special class of submersion, namely \textit{locally trivial morphisms} $p : \sG \to X$. This will sometimes simplify our approach, in particular in Section \ref{GpBdleTrivByCover} (as well as play a role in the proof of Theorem \ref{ClassifGpCovering}) while still covering all the examples we need for our purposes.

\subsection{Locally trivial group spaces}

Let $X$ be a complex analytic manifold. In these notes, \emph{a group bundle} on $X$ will be a group space which is locally trivial in the following sense.

\begin{definition}\label{defGrpBdle}
A \emph{group bundle} on a complex analytic manifold $X$ is a complex analytic manifold $\mathscr{G}$ together with a holomorphic map $\sG\overset{p}{\longrightarrow}X$, satisfying the following conditions:
\begin{enumerate}
    \item For all $x\in X$, the fibre $\sG_x := p^{-1}(x)$ is a complex Lie group.
    \item For all $x\in X$, there exists an open neighborhood $U\subseteq X$ containing $x$ and an isomorphism of complex analytic manifolds 
        $$ \Phi_U : p^{-1}(U)\to U\times p^{-1}(x) $$
    such that $\mathrm{pr}_U \circ \Phi_U = p|_{p^{-1}( U )}$ and, for all $y \in U$, the induced bijection 
        $$ \Phi_U|_{p^{-1}(y)}: p^{-1}( y ) \lra p^{-1}( x )$$
    is an isomorphism of complex Lie groups. 
\end{enumerate}
\end{definition}

In addition to $\sG_x := p^{-1}(x)$, we shall often use the notation $\sG|_U := p^{-1}(U)$. One consequence of the local triviality condition is that, over a given connected component of $X$, the fibres of $\sG$ are pairwise isomorphic. Let us now look at a few examples of group bundles.

\begin{example}\label{exampleConstantGroup}
    Let $G$ be a complex Lie group, then the complex analytic manifold $G_X := X \times G$, together with the projection onto the first factor, forms a group bundle over $X$, called the \emph{constant group} over $X$ associated to $G$.
\end{example}

\begin{example}\label{exampleConstantGroup1}
    Let $E$ be a holomorphic vector bundle of rank $n$ over a complex analytic manifold $X$. Then we can view $E$ as a group bundle over $X$, with fibres isomorphic to the additive group $(\C^n, +)$. This group bundle is \textit{non-constant} in general (meaning that $E$ is not a trivial bundle).
\end{example}

To construct further non-constant examples, we can start from a constant group on some space $Y$ equipped with a free action of a finite group $\pi$ and take the quotient of $Y \times G$ by $\pi$, where $G$ is a complex Lie group on which $\pi$ acts by automorphisms (in particular, to the left). Then, if $\pi$ acts non-trivially on $G$, the group bundle $\pi \backslash (Y\times G)$ will be non-constant on $\pi \backslash Y$.

\begin{example}\label{GpBdleFromCov}
Let $q : Y \lra X $ be a covering map of degree $2$ between Riemann surfaces. Let $\pi := \Aut_X(Y) \simeq \Z / 2\Z $ be the automorphism group of that covering, and let $\pi$ act to the left on $ G := \GL(n,\C) $ via the involution $ g \lmt \,^t g^{-1}$. Then the quotient space $$ \sG := \pi \backslash ( Y \times G ) $$ is a group bundle on $ X $.
\end{example}

Other examples of non-constant group bundles on $X$ can be constructed starting from non-trivial bundles on $X$. We recall that, if $G$ is a complex analytic Lie group and $X$ is a complex analytic manifold, a principal $G$-bundle on $X$ is a complex analytic manifold $P$ equipped with a \textit{free} right $G$-action, a $G$-invariant morphism of complex manifolds $p : P \to X$ and local trivialisations $p^{-1}(U) \simeq U \times G$ that are $G$-equivariant with respect to the induced $G$-action on the $G$-invariant open set $p^{-1}(U)$ and the right $G$-action $(x, h) \cdot g := (x, hg)$ on $U \times G$. By definition, the morphisms of principal bundles over $X$ are $G$-equivariant bundle maps over $X$.

\begin{example}\label{Adjoint_of_ppal_bundle}
Let $G$ be a complex Lie group and let $\pi : P \longrightarrow X$ be a principal $G$-bundle. Let $G$ act to the left on $P \times G$ via 
    $$ g \cdot (p,h) := (p \cdot g^{-1}, g h g^{-1}) .$$
Then the fibre bundle $\text{Ad}(P) := G \backslash (P\times G)$
is a group bundle on $ X $.
\end{example}

Note that the global sections of the group bundle $\Ad(P)$ form a group, isomorphic to the group of automorphisms of the principal $G$-bundle $P$, meaning that we have a group isomorphism $\Gamma(\text{Ad}(P))\simeq \Aut(P)$ where $\Aut(P)$ is the so-called \emph{gauge group} of the principal bundle $P$. Indeed, since $P \times G$ is the pullback of $\Ad(P)$ along $\pi : P \to X$, a global section of $\Ad(P)$ can be seen as a map $F: P \lra G$ satisfying the $G$-equivariance property
$$
\forall p \in P,\ \forall g \in G,\ F(p\cdot g)=g^{-1}F(p) g.
$$
And this defines an automorphism of principal $G$-bundles $\Phi_F : P \lra P$ via the formula 
$$
\Phi_F(p) = p\cdot F(p)
$$
because the $G$-equivariance of $F$ implies that
$$
\Phi_F(p \cdot g) = (p \cdot g) \cdot F(p \cdot g) = (p \cdot g) \cdot (g^{-1} F(p) g) = \big(p \cdot F(p)\big) \cdot g = 
\Phi_F(p) \cdot g.
$$
\begin{exercise}
Show that, conversely, an automorphism $\Phi \in \Aut(P)$ defines a global section of $\Ad(P)$ and that this indeed defines a group isomorphism $\Gamma(\text{Ad}(P))\simeq \Aut(P)$.
\end{exercise}

\begin{example}\label{groupBundleExampleAutomorphismFibreBundle}
More generally, let us consider the \emph{automorphism bundle} of a fibre bundle $E$ on $X$. Let $E$ be a fibre bundle on $X$ (e.g.\ a vector bundle) and consider the set 
$$
\mathcal{A}ut(E) := \bigsqcup_{x\in X} \Aut(E_x)
$$
where $\Aut(E_x)$ is the automorphism group of the fibre $E_x$.
Then $\mathcal{A}ut(E)$ is a group bundle on $X$. The group of global sections of $\mathcal{A}ut(E)$ is (isomorphic to) the automorphism group of $E$, meaning that $\Gamma(\mathcal{A}ut(E)) \simeq \Aut(E)$.
\end{example}

Let us give an example where $E$ is a vector bundle on $X$. Recall first that, if $E$ is a vector bundle of rank $n$ represented by a cocycle of transition functions 
$$
g_{ij} : U_i \cap U_j \lra \GL(n,\C),
$$
then local sections of $E$ over some open set $V \subseteq X$, can be represented by holomorphic functions
$$
s_i : V \cap U_i \lra \C^n \ \text{such that} \ s_i = g_{ij} s_j
$$
over $V \cap (U_i \cap U_j)$. So local sections of the group bundle $\mathcal{A}ut(E)$ over $V$ can be represented by holomorphic maps 
$$
u_i : V \cap U_i \lra \GL(n,\C) \ \text{such that} \ 
u_i = g_{ij} u_j g_{ij}^{-1}
$$
over $V \cap (U_i \cap U_j)$. Let now $L$ be a complex line bundle on $X$ and consider the vector bundle
\begin{equation}\label{bdle_with_fixed_det}
    E := \C_X^{n-1} \oplus L
\end{equation}
which is the direct sum of $L$ and the trivial line bundle of rank $n-1$ for some integer $n \geq 2$. Then the local sections of the group bundle $\mathcal{A}ut(E)$ can be described as follows. We fix an open cover 
$(U_i)_{i\in I}$ of $X$ that trivializes the line bundle $L$ and a cocycle of transition functions
$
h_{ij}:U_i \cap U_j \to\C^*
$
representing the line bundle $L$. Then $E|_{U_i}$ is trivial and we can represent $E$ by the cocycle of transition functions
$g_{ij}: U_i \cap U_j \to \text{GL}(n,\mathbb{C})$ defined by 
\begin{equation}\label{cocycle_bdle_with_fixed_det}
    g_{ij}
    =
    \begin{pmatrix}
        1 & 0 & 0& \ldots & 0 \\
        0 & 1 & 0 & \ldots & 0 \\
        \vdots &  & \ddots &  &\vdots \\
        0&\ldots&&1&0  \\
        0 &0&\ldots&0&h_{ij} 
    \end{pmatrix}
    .
\end{equation}
As a consequence, a local section $u$ of the group bundle $\mathcal{A}ut(\C_X^{n-1}\oplus L)$, over some open set $V \subseteq X$, corresponds to a collection of holomorphic maps 
$$
u_i : V \cap U_i \to \GL(n,\C)\
\text{such that}\
u_i = g_{ij} u_j g_{ij}^{-1}
$$
over $V \cap (U_i \cap U_j)$.

\begin{remark}
Definition \ref{defGrpBdle} admits variants with small modifications. For instance, given a morphism of complex analytic manifolds $p : \sG \lra X$ whose fibres are complex analytic Lie groups, the local triviality condition can also be formulated as follows: for all $x \in X$, there exists an open neighbourhood $U\subseteq X$ of $x$, a complex analytic Lie group $G_U$ and an isomorphism of complex analytic fibre bundles $p^{-1}(U) \simeq U \times G_U$ such that all the induced bijections $p^{-1}(y) \simeq G_U$ are group morphisms. If one does not assume that the fibres of $p$ are Lie groups, then one can instead require that a change of local trivialisation $(U \cap V) \times G_U \simeq (U \cap V) \times G_V$ (over $U \cap V$) be a group morphism fibrewise. Then the fibres of $p$ acquire a well-defined structure of complex analytic Lie group.
\end{remark}

\subsection{Morphisms, subgroups and Lie algebras}

Let us now define morphisms of group bundles and introduce the notion of subgroup of a group bundle.

\begin{definition}
    Let $\sG$, $\sG'$ be group bundles over $X$. A \emph{morphism of group bundles} between $\sG$ and $\sG'$ is a complex analytic map $F:\sG \lra \sG'$ such that:
    \begin{enumerate}
        \item The diagram 
        \begin{center}
            \tikzcdset{arrow style=tikz, diagrams={>=angle 90}}
            \begin{tikzcd}
                \sG \arrow[to=1-3,"F"] \arrow[to=2-2,"p'",']& & \sG' \arrow[to=2-2,"p"]\\
                &X&  
            \end{tikzcd}
        \end{center}
        commutes.
        \item For all $x\in X$, the restriction $F_x:=F|_{\sG_x}:\sG_x\longrightarrow\sG'_x$ is a homomorphism of complex Lie groups. 
    \end{enumerate}
\end{definition}

\begin{example}\label{det_morphism}
    Let $E$ be a vector bundle on $X$ and let $L := \det(E)$ be its determinant line bundle. The automorphism bundle of a line bundle is the constant group $\C^*_X$ and the determinant map $\det_x : \Aut(E_x) \lra \C^*$ induces a morphism of group bundles $\det : \mathcal{A}ut(E) \lra \C^*_X$.
\end{example}
So a morphism of group bundles is basically a family of group morphisms which, taken together, define a morphism of complex analytic manifolds. Similarly, a subgroup $\sH$ of a group bundle $\sG$ is basically a family of subgroups that form a sub-bundle of $\sG$ (the latter condition being a requirement on the \textit{compatibility} of the local trivialisations of $\sG$ and $\sH$).

\begin{definition}\label{subgroup}
	Let $\sG$ be a group bundle on $X$. A subset $\sH \subseteq \sG$ is called a \emph{subgroup} of $\sG$ if it satisfies the following two conditions:
    \begin{enumerate}
        \item For all $x \in X$, the subset $\sH_x := \sG_x \cap \sH$ is a Lie subgroup of $\sG_x$.
        \item For all $x \in X$, there exists an open neighborhood $U \subseteq X$ containing $x$ and a local trivialisation 
        $$ \Phi:\sG|_U \to U \times \sG_x $$
        of $\sG$ that induces a bijection
        $$\sH \cap \sG|_U \simeq U \times \sH_x .$$  
    \end{enumerate}
    A trivialising open set $U$ for $\sG$ that satisfies Condition (2) above will be called a trivialising open set for the subgroup $\sH$.
\end{definition}

As a consequence of Definition \ref{subgroup}, we see that $\sH$ is a submanifold of $\sG$ and that $p|_{\sH}: \sH \lra X$ is a group bundle. Note that, as usual for subbundles, we are not saying that $\sH_U$ is trivial for all trivialising open set of $\sG$, only that there exists a local trivialisation of $\sG$ that is adapted to $\sH$ in the sense of Condition (2) in Definition \ref{subgroup}.

\begin{example}
	Building on Example \ref{exampleConstantGroup}, let $H$ a complex analytic subgroup of $G$. Then the constant group $H_X := X\times H\lra X$ is a subgroup of the group bundle $G_X$. Similarly, in Example \ref{exampleConstantGroup1}, take $F$ to be a sub-bundle of the vector bundle $E$. Then we can view $F$ as a subgroup of the additive group bundle associated to $E$.
\end{example}

To construct more examples, we observe that the kernel of a bundle morphism $\Phi : \sG \lra \sG'$ that is of constant rank as a morphism of complex analytic manifolds, defines a subgroup $\sH := \ker \Phi$.

\begin{example}\label{exampleSubgroupSL}
    Let $E$ be a vector bundle on $X$. Its automorphism bundle $\mathcal{A}ut(E)$ can also be denoted by $\GL(E)$. Then, thanks to the construction of the determinant morphism $\det : \GL(E) \to \C_X^*$ in Example \ref{det_morphism}, we have a well-defined short exact sequence of group bundles
    \begin{equation}\label{det_one_automorphisms}
    1 \lra \SL(E) \lra \GL(E) \overset{\det}{\lra}\C^*_X \lra 1
    \end{equation}
    where $\SL(E) := \ker (\det)$ is the group bundle whose fibres can be described as follows:
    $$
    \SL(E)_x = \{ A \in \GL(E_x)\ |\ \det(A) = 1 \} = \SL(E_x).
    $$
\end{example}

\begin{exercise}
    Define normal subgroups of a group bundle $\sG$. Given a subgroup $\sH \subseteq \sG$, show that there is a fibre bundle $\sG / \sH$ and that, if $\sH$ is a normal subgroup, this fibre bundle is a group bundle.
\end{exercise}

Our final task in this subsection is to introduce the Lie algebra of a group bundle $\sG$ over $X$. To that end, note first that as a consequence of Definition \ref{defGrpBdle}, we can reconstruct a group bundle $p : \sG \lra X$ from its local trivialisations, by gluing two different local trivialisations $p^{-1}(U) \simeq U \times \sG_x$ and $p^{-1}(V) \simeq V \times \sG_x$ at $x$ via the identification
\[
\Phi_U \circ \Phi_V^{-1} |_{p^{-1}(U \cap V)} : 
    (U \cap V) \times \sG_x \overset{\simeq}{\lra} (U \cap V) \times \sG_x 
\]
which can be written as
$$ \Phi_U \circ \Phi_V^{-1}(y, g) = \big(y, h_{UV}(y) \cdot g\big)$$
where $h_{UV} : U \cap V \to \Aut(\sG_x)$ is a morphism of complex analytic manifolds. In particular, for all $y\in U\cap V$, the isomorphism of complex analytic Lie groups $h_{UV}(y) : \sG_x \to \sG_x$ induces an isomorphism of complex Lie algebras $\Lie(\sG_x) \to \Lie(\sG_x)$, which we will also denote by $h_{UV}(y)$. Then the Lie algebra bundle $\Lie(\sG)$ of the group bundle $\sG$ is by definition the bundle of Lie algebras obtained by gluing $U \times \Lie(\sG_x)$ and $V \times \Lie(\sG_x)$ via the induced identification 
\[
    \begin{array}{rcl}
        (U \cap V) \times \Lie(\sG_x) & \overset{\simeq}{\lra} & (U \cap V) \times \Lie(\sG_x ) \\
        (y, v) & \lmt & (y, h_{UV}(y)\cdot v)
    \end{array}\]
and sometimes we will simply call it the \textit{Lie algebra} of the group bundle $\sG$.

\begin{example}
    Let $E$ be a vector bundle over $X$. The Lie algebra of the group bundle $\mathbf{SL}(E)$ constructed in \eqref{det_one_automorphisms} is the Lie algebra subbundle $\mathfrak{sl}(E) \subset \End(E)$ whose fibres can be described as follows
    $$
    \mathfrak{sl}(E)_x := \{ A \in \End(E_x) \ | \ \mathrm{tr}(A) = 0 \} = \mathfrak{sl}(E_x).
    $$
\end{example}

\begin{definition}
    Let $p : \sG \to X$ be a group bundle. The Lie algebra bundle of $\sG$ will be denoted by $\Lie(\sG)$.
\end{definition}

\subsection{Group bundles trivialised by a cover}\label{GpBdleTrivByCover}

Generalising the situation of Example \ref{GpBdleFromCov}, we fix a Galois cover $q : Y \to X$, with automorphism group $\pi := \Aut_X(Y)$. Recall that a Galois cover means a connected cover $q : Y \to X$ such that the induced covering $\overline{q} : \pi\backslash Y \to X$ is an isomorphism, which is equivalent to $\pi$ acting transitively on the fibres of $q$. Group bundles $\sG \to X$ whose pullback to $Y$ is trivial are called $\pi$-bundles. Such bundles are necessarily locally trivial on $X$. If we also fix the typical fibre $G$, we speak of $(\pi, G)$-group bundles (expanding on Seshadri's terminology in \cite{Seshadri_parab_remarks}). We recall that the pullback bundle of a fibre bundle $E \to X$ by a map $q : Y \to X$ is by definition the fibre product $$q^* E := Y \times_X E\,.$$ As such, $q^*E$ is itself a fibre bundle over $Y$, satisfying the universal property of a fibre product, which is usually depicted as follows.

\begin{equation*}\label{univ_ppty_pullback}
    \begin{tikzcd}
        Z \arrow[to=2-3, bend left, "u"] \arrow[to=2-2, dashed, "\exists ! \psi"] \arrow[to=3-2, bend right, "v"] & & \\
        & q^*E \arrow[to=2-3, "\hat{q}"] \arrow[to=3-2, "\hat{p}"] & E \arrow[to=3-3,"p"] \\
        & Y \arrow[to=3-3, "q"] & X
    \end{tikzcd}    
\end{equation*}

\begin{definition}\label{pi_G_gp_bundle}
    Given a complex analytic manifold $X$, a Galois cover $q : Y \to X$, with automorphism group $\pi := \Aut_X(Y)$, and a complex Lie group $G$, a group bundle $\sG \to X$ such that $q^* \sG \simeq Y \times G$ as group bundle will be called a $(\pi, G)$-group bundle.
\end{definition}

The interest of this definition is that we can classify all $(\pi, G)$-group bundles. Indeed, the pullback bundle $q^* \sG$ is not just a group bundle with fibre $G$: it has a canonical $\pi$-equivariant structure,  meaning that the action of $\pi$ on $Y$ lifts to an action $q^*\sG$, respecting the group structure on the fibres. More precisely, we will show the existence of a collection of maps $(\tau_f)_{f\in\pi}$ satisfying the following conditions:

\begin{enumerate}
    \item $\tau_{\mathrm{id}_{Y}} = \mathrm{id}_{q^*\sG}$ and the following diagram is commutative:
    \begin{equation}\label{action_lift}
        \begin{tikzcd}
            q^*\sG \arrow[to=1-2,"\tau_f"] \arrow[to=2-1] & q^*\sG \arrow[to=2-2]\\
            Y \arrow[to=2-2,"f"] & Y
        \end{tikzcd}
    \end{equation} 
    \item For all $f_1, f_2 \in \pi$, 
    \begin{equation}\label{lift_compatible_with_product}
        \tau_{f_1 \circ f_2} = \tau_{f_1} \circ \tau_{f_2}.
    \end{equation}
    \item For all $f\in \pi$, all $y \in Y$ and all $g_1, g_2 \in (q^*\sG)_y$, 
    \begin{equation}\label{gp_morphism_ppty_equiv_struc}
    \tau_f(g_1 g_2) = \tau_f(g_1) \tau_f (g_2).
    \end{equation}
\end{enumerate}

\noindent So if $q^*\sG$ is isomorphic to $Y \times G$ as a group bundle, we inherit a $\pi$-equivariant structure on $Y \times G$ and these are the structures that we will end up classifying (Theorem \ref{ClassifGpBdleTrivByCover}).

\smallskip

The reason why the maps $(\tau_f)_{f \in \pi}$ exist and satisfy the required properties \eqref{action_lift}, \eqref{lift_compatible_with_product} and \eqref{gp_morphism_ppty_equiv_struc} is a direct consequence of the universal property of a pullback bundle and the fact that if $f\in\pi$ is an automorphism of the cover $q : Y \to X $, then by definition $q\circ f = q$. Given such an $f$, the map $\tau_f$ is then defined by means of the following commutative diagram:
\begin{equation}\label{def_action_lift}
    \begin{tikzcd}
        q^* \sG \arrow[to=2-3, bend left, "\hat{q}"] \arrow[to=2-2, dashed, "\exists ! \tau_f"] \arrow[to=3-1,"\hat{p}"] & & \\
        & q^*\sG \arrow[to=2-3, "\hat{q}"] \arrow[to=3-2, "\hat{p}"] & \sG \arrow[to=3-3,"p"] \\
        Y \arrow[to=3-2, "f"] & Y \arrow[to=3-3, "q"] & X
    \end{tikzcd}    
\end{equation}

\noindent We thus have a left action of $\pi$ on 
$$
q^*\sG = \big\{ (\xi,g) \in Y \times \sG \ | \ q(\xi) = p(g) \big\}
$$ 
defined, for all $(\xi,g) \in q^*\sG$ and all $f\in \pi$ by 
$$
f \cdot (\xi,g) = \tau_f(\xi, g).
$$
By construction, one has $\pi \backslash q^*\sG \simeq \sG$ over $X \simeq \pi \backslash Y$, where $\pi$ acts on $q^* \sG$ via the family $(\tau_f)_{f \in \pi}$.

\smallskip

In what follows, we denote by $\Aut(G)$ the group of complex analytic, group automorphisms of the complex Lie group $G$. The set $\Maps(Y,\Aut(G))$, consisting of holomorphic maps $\lambda : Y \to \Aut(G)$, is a group under pointwise composition of such maps: 
\begin{equation}\label{product_on_mapping_space_to_Aut(G)}
    (\lambda_1 \cdot \lambda_2)(y) := \lambda_1(y) \circ \lambda_2(y).
\end{equation}
Moreover, the group $\pi = \Aut_X(Y)$ acts to the right by group automorphisms on the group 
$\Maps(Y,\Aut(G))$, via 
\begin{equation}\label{action_of_pi_on_Aut_G}
    \lambda^{f} := \lambda \circ f.
\end{equation} 
So there is a notion of \emph{crossed morphism} 
$\phi : \pi \to \Maps(Y, \Aut(G))$, which we now introduce formally.
\begin{definition}
    A crossed morphism $\phi : \pi \to \Maps(Y, \Aut(G))$ with respect to the $\pi$-action defined in Equation \ref{action_of_pi_on_Aut_G} is a map satisfying 
    \begin{equation}\label{crossed_morphism_eq}
        \phi_{f_1 \circ f_2} = \phi_{f_1}^{f_2} \cdot \phi_{f_2}
    \end{equation} 
    where we denote by $\phi_f : Y \to \Aut(G)$ the image of $f \in \pi = \Aut_X(Y)$ under $\phi$.
\end{definition} 
We will denote by $Z^1(\pi, \Maps(Y, \Aut(G)))$ the set of crossed morphisms $\phi : \pi \to \Maps(Y,\Aut(G))$. Then, the element $\lambda \in \Maps(Y, \Aut(G))$ acts on $\phi \in Z^1(\pi, \Maps(Y, \Aut(G)))$ via the formula (with $f \in \pi$ and $y \in Y$)
\begin{equation}\label{action_on_crossed_morphisms}
    (\lambda \bullet \phi)_f (y) := \lambda \big( f(y) \big) \circ \phi_f\big(y\big) \circ \lambda \big( y \big)^{-1}\ .
\end{equation}

\begin{exercise}
    Check that the map $\lambda \bullet \phi : \pi \to \Maps(Y,\Aut(G))$ defined in \eqref{action_on_crossed_morphisms} is indeed a crossed morphism. Show  moreover that $\mathbf{1} \bullet \phi = \phi$ (where $\mathbf{1} : Y \to G$ is the constant function equal to $\mathrm{Id}_G$ on $Y$) and that $(\lambda_1 \cdot \lambda_2) \bullet \phi = \lambda_1 \bullet (\lambda_2 \bullet \phi)$, so that we indeed have a left action of the group $\Maps(Y, \Aut(G))$ on $Z^1(\pi, \Maps(Y, \Aut(G)))$.
\end{exercise}

We then have the following classification result for $(\pi, G)$-group bundles.

\begin{theorem}\label{ClassifGpBdleTrivByCover}
    Assume that $X$ is connected. Let $q : Y \to X$ be a Galois cover, with automorphism group $\pi := \Aut_X(Y)$, and let $G$ be a complex Lie group.
    \begin{enumerate}
        \item If $p : \sG \to X$ is a $(\pi, G)$-group bundle, then there exists a crossed morphism (in the sense of Equation \eqref{crossed_morphism_eq})
        \begin{equation}\label{crossed_morphism_from_pi_to_maps_to_G}
            \varphi : \pi \lra \Maps\big(Y, \Aut(G)\big)
        \end{equation}
        and an isomorphism of group bundles 
        $$
        \sG \simeq \pi \backslash \big(Y \times G\big)
        $$ 
        where $f \in \pi$ acts to the left on $(y, g) \in Y \times G$ via
        \begin{equation}\label{twisted_action_of_pi_on_Y_times_G}
            \tau_f (y,g) := \big(f(y), \varphi_f(y) (g)\big).
        \end{equation}
        \item Moreover, if $\phi$ and $\phi'$ are crossed morphisms from $\pi$ to $\Maps(Y, \Aut(G))$, then the associated group bundles $\sG$ and $\sG'$ are isomorphic over $X$ if and only if $\phi$ and $\phi'$ are conjugate under the left action of $\Maps(Y, \Aut(G))$ defined in \eqref{action_on_crossed_morphisms}.
    \end{enumerate} 
    So, the set of isomorphism classes of $(\pi,G)$-group bundles is in bijection with the group cohomology set 
    $$
    H^1\big(\pi, \Maps\big(Y, \Aut(G)\big) \big) := 
    \Maps\big(Y, \Aut(G) \big) \ \big\backslash \ Z^1\big(\pi, \Maps\big(Y, \Aut(G)\big) \big)\ .
    $$
\end{theorem}

\begin{proof}
    Let $\sG$ be a $(\pi, G)$-group bundle on $X = \pi \backslash Y$, in the sense of Definition \ref{pi_G_gp_bundle}. Then the $\pi$-equivariant group bundle $q^*\sG$ is isomorphic to $Y \times G$ as a group bundle on $Y$. By construction, there is an isomorphism of group bundles $\sG \simeq \pi \backslash q^*\sG$, where $\pi$ acts on $q^*\sG$ via the $\pi$-equivariant structure \eqref{def_action_lift}. So, to prove the first claim of the theorem, we only need to classify all $\pi$-equivariant structures on the group bundle $Y \times G$. Namely, it suffices to show the existence of a crossed morphism $\phi$ as in 
    \eqref{crossed_morphism_from_pi_to_maps_to_G} such that the $\pi$-equivariant structure of $Y \times G$ induced by the isomorphism with $q^*\sG$ is given by Formula 
    \eqref{twisted_action_of_pi_on_Y_times_G}. So let $(\tau_f)_{f \in \pi}$ be a $\pi$-equivariant structure on $Y \times G$. By commutativity of Diagram \eqref{action_lift}, this structure $(\tau_f)_{f\in\pi}$ must be of the form
    \begin{equation}\label{equiv_structure_on_product_bundle}
    \tau_f : Y \times G \lra Y \times G, \quad
    (y, g) \lmt \big( f(y), \Phi_f(y,g) \big)
    \end{equation}
    where, in view of Property \eqref{gp_morphism_ppty_equiv_struc}, the map $\Phi_f : Y \times G \to G$ satisfies
    \begin{equation}
        \Phi_f(y, g_1 g_2) = \Phi_f(y,g_1)\Phi_f(y,g_2).
    \end{equation}
    Finally, due to Property \eqref{lift_compatible_with_product}, we have, for all $f_1, f_2 \in \pi$, 
    \begin{equation}\label{lift_compatible_with_product_on_product_bundle}
        \Phi_{f_1 \circ f_2} (y, g) = \Phi_{f_1}\big( f_2(y), \Phi_{f_2} (y, g) \big).
    \end{equation}
    Thus, for all $f \in \pi$, there is an induced map
    \begin{equation}\label{towards_crossed_morphism}
        \phi_f: Y \to \Aut(G),\quad y \lmt \phi_f(y) := \Phi_f(y,\,\cdot\,)
    \end{equation}
    and this defines a map
    $$
    \phi : \pi \lra \Maps\big( Y, \Aut(G) \big)
    $$
    which, in view of \eqref{lift_compatible_with_product_on_product_bundle}, satisfies 
    $$
    \phi_{f_1 \circ f_2} = \phi_{f_1}^{f_2} \cdot \phi_{f_2}
    $$
    as defined in \eqref{crossed_morphism_eq}. This means that $\phi : \pi \to \Maps(Y,\Aut(G))$ is a crossed morphism with respect to the $\pi$-action on $\Maps(Y,\Aut(G))$ defined in \eqref{action_of_pi_on_Aut_G}. Indeed, using the definition of $\phi_f(y)\in \Aut(G)$ given in \eqref{towards_crossed_morphism} as well as the group structure \eqref{product_on_mapping_space_to_Aut(G)} on $\Maps(Y,\Aut(G))$, we have, for all $y \in Y$ and all $g\in G$,
    $$
    \phi_{f_1\circ f_2}(y)(g) = \phi_{f_1}^{f_2 }\big(y\big) \big( \phi_{f_2}(y) (g) \big) = \big( \phi_{f_1}^{f_2}(y) \circ \phi_{f_2}(y) \big) (g) = (\phi_{f_1}^{f_2} \cdot \phi_{f_2})(y) (g).
    $$
    This concludes the proof of the first part of Theorem \ref{ClassifGpBdleTrivByCover}. 
    
    \smallskip

    To prove the second part, observe that the automorphism group of the group bundle $G_Y := Y \times G$ satisfies $\Aut(G_Y) \simeq \Maps(Y, \Aut(G))$. So if the $(\pi,G)$-group bundles $\sG$ and $\sG'$ (defined by the crossed morphisms $\phi$ and $\phi'$) are isomorphic as group bundles, there is an induced isomorphism of $\pi$-equivariant group bundles which we denote by $\alpha$, going from $q^* \sG \simeq G_Y$ to $q^* \sG' \simeq G_Y$, hence a map $\lambda \in \Maps(Y, \Aut(G))$ that commutes to the $\pi$-equivariant structures $\tau$ and $\tau'$ induced respectively by $\phi$ and $\phi'$. Concretely, this means that the automorphism $\alpha : G_Y \to G_Y$ defined by $(y, g) \mapsto (y, \lambda(y)(g))$ satisfies, for all $f \in \pi$, the relation 
    $$
    \alpha \circ \tau_f \circ \alpha^{-1} = \tau'_f
    $$
    hence, for all $y \in Y$ and all $g \in G$,
    $$
    \big( f(y), \big(\lambda(f(y)) \circ \phi_f(y) \circ \lambda(y)^{-1}\big) (g)\big) = \big(f(y), \phi'_f(y)(g)\big),
    $$
    which proves that $\phi' = \lambda \bullet \phi$, where $\lambda \bullet \phi$ is defined as in \eqref{action_on_crossed_morphisms}.
\end{proof}

\begin{remark}\label{grpBundlesWithDiscreteFibres}
    If $G$ is replaced by a \emph{discrete} group $\Gamma$, then the map $\Phi_f : Y \times \Gamma \to \Gamma$ defined by \eqref{equiv_structure_on_product_bundle} is in fact constant in $Y$, so the induced map $\phi_f$, defined in \eqref{towards_crossed_morphism}, is just an automorphism of $G$. As a consequence, in this case we have a group morphism $\phi : \pi \to \Aut(\Gamma)$.
\end{remark}

Theorem \ref{ClassifGpBdleTrivByCover} applies in particular to group bundles $\sG$ over $X$ that become trivial after pullback to the universal cover of $X$ (in which case $\pi \simeq \piX$). There are two noteworthy situations in which this occurs:
\begin{enumerate}
    \item When the total space of the universal cover $q : \Xt \to X$ is a contractible Stein manifold \cite{Grauert_hol_bundles}. 
    \item When the group bundle $p : \sG \to X $ has discrete fibres.
\end{enumerate}

The second case will be analysed in Theorem \ref{ClassifGpCovering}. As for the first case, since a group bundle $\sG$ is assumed, by definition, to be locally trivial over $X$, the holomorphic fibre bundle $q^* \sG$ is a locally trivial fibre bundle over $\Xt$, and since $\Xt$ is contractible, this implies that $q^* \sG$ is topologically trivial. Then, as $\Xt$ is a Stein manifold, a topologically trivial holomorphic bundle is in fact holomorphically trivial. So if we fix a base point in $\Xt$ and denote by $G$ the fibre of $q^*\sG$ over that point, we get a complex analytic Lie group $G$ and a group bundle isomorphism $q^*\sG \simeq \Xt \times G$. This means that when $\Xt$ is a contractible Stein manifold, every group bundle on $X$ is a $(\piX, G)$-bundle for some complex analytic Lie group $G$. This happens for instance when $X$ is a compact Riemann surface of genus $g_X \geq 1$ (so the universal cover of $X$ is a contractible open subset of the Riemann sphere), or when $X$ is a ball quotient in higher dimension (meaning of the form $\Gamma\backslash B^n$, where $B^n\subset \C^{n+1}$ is the open unit ball in $\C^{n+1}$ and $\Gamma < \mathbf{PU}(n,1)\simeq \Aut(B^n)$ is a discrete subgroup).

\smallskip 

Finally, as an immediate corollary of Theorem \ref{ClassifGpBdleTrivByCover}, we obtain that the Lie algebra of a $(\pi, G)$-group bundle $\sG$ on $X = \pi \backslash Y$ is the quotient by $\pi$ of the product bundle $Y \times \Lie(G)$.

\begin{corollary}
    Under the assumptions of Theorem \ref{ClassifGpBdleTrivByCover}, the Lie algebra of the group bundle $\sG \lra X$ satisfies $$\Lie(\sG) \simeq \pi \backslash \big(Y \times \Lie(G)\big)$$ where $f \in \pi$ acts to the left on $(y, v) \in Y \times \Lie(G)$ via $$f \cdot (y, v) := \big(f(y), \varphi_f(y)(v)\big),$$ the Lie algebra automorphism $v \lmt \varphi_f(y)(v)$ being, for all $y \in Y$, the derivative at the identity of the Lie group automorphism $g \lmt \varphi_f(y) (g)$.
\end{corollary}

\begin{proof}
    Follows from the fact that $q^*(\Lie(\sG)) \simeq \Lie(q^*\sG) \simeq \Lie(Y \times G) \simeq Y \times \Lie(G)$.
\end{proof}

\section{Torsors}

In this section, we start with a general discussion about families of principal homogeneous spaces parameterised by an abstract topological space. Then we classify torsors under a group space over a complex analytic manifold. Finally, we show in what sense holomorphic vector bundles with non-trivial determinant correspond to torsors under a non-constant group.

\subsection{Principal homogeneous spaces}

Let $X$ be a topological space and let $\sG \to X$ be a group object in the category of spaces over $X$ (see Section \ref{GpSpace}). Intuitively, a right $\sG$-space is a space over $X$ on which $\sG$ acts fibrewise to the right. This means a pair $(\pi, \alpha)$ where
\begin{itemize}
    \item $\pi : \sP \to X$ is a space over $X$, and
    \item $\alpha : \sP \times_X \sG \to \sG$ is a morphism of spaces over $X$,
\end{itemize}
such that, for all $x \in X$, the restricted map 
$$
\alpha_x : \sP_x \times \sG_x \lra \sP_x 
$$
is a right group action in the category of topological spaces. The element $\alpha(p, g) \in \sP$, will often be denoted simply by $p \cdot g$. A morphism of right $\sG$-spaces between $\sP$ and $\sP'$ is a morphism $u : \sP \to \sP'$ of spaces over $X$ (meaning that $\pi' \circ u = \pi$, where $\pi' : \sP' \to X$) which is $\sG$-equivariant (meaning that $u(p \cdot g) = u(p) \cdot g$). Left $\sG$-spaces and morphisms between them are defined similarly.

\smallskip

Recall that the action of the group $\sG_x$ on the space $\sP_x$ is called \textit{transitive} if it has only one orbit. This is equivalent to asking that the map
\begin{equation}\label{charac_torsors}
    \begin{array}{rcl}
        \sP_x \times \sG_x & \lra & \sP_x \times \sP_x \\
        (p, g) & \lmt & (p, p \cdot g)
    \end{array}
\end{equation}
be surjective. In that case, the $\sG_x$-space $\sP_x$ is called a \textit{homogeneous} $\sG_x$-space. Similarly, the action is called \textit{free} if all stabilizers are trivial, and this is equivalent to asking that the map \eqref{charac_torsors} be injective. When the action is both transitive and free, the $\sG_x$-space $\sP_x$ is called a \textit{principal homogeneous} $\sG_x$-space, or $\sG_x$-\textit{torsor}. If $\sP_x$ is non-empty, this is equivalent to asking that $\sP_x$ be isomorphic, as a right $\sG_x$-space, to the space $\sG_x$ endowed with the action of $\sG_x$ on itself by right translations. Indeed, if $\sP_x$ is a principal homogeneous $\sG_x$-space, then for all $p \in \sP_x$, the map sending $g \in \sG_x$ to $p \cdot g$ is an isomorphism of right $\sG_x$-spaces. The discussion in this paragraph is purely punctual and we now want to discuss \textit{families} of principal homogeneous spaces under a group space $\sG \to X$.

\begin{definition}\label{torsor_def}
    Let $X$ be a topological space and let $\sG \to X$ be a group space over $X$. A right $\sG$-\textit{torsor} is a right $\sG$-space $\sP$ over $X$ such that, for all $x \in X$, there is an open neighbourhood $U$ of $x$ and an isomorphism of $\sG|_U$-spaces $\sG|_U \simeq \sP|_U$.
\end{definition}

\begin{example}\label{trivial_torsor}
Let $\sG$ be a group space over $X$. Then $\sG$ acts on itself by right translations and the identity map is an isomorphism of right $\sG$-spaces over $X$. We call this the \emph{trivial $\sG$-torsor} over $X$.
\end{example}

As a consequence of Definition \ref{torsor_def}, if $\sP$ is a right $\sG$-torsor over $X$, then $\sP$ \textit{admits local sections} over any open set $U \subset X$ for which there exists an isomorphism of right $\sG|_U$-spaces $\sG|_U \simeq \sP|_U$. So, a right $\sG$-torsor can be defined equivalently as a right $\sG$-space $q : \sP \to X$ such that:
\begin{enumerate}
    \item $q$ admits local sections, and
    \item the map $\psi : \sP \times_X \sG \to \sP \times_X \sP $, defined fibrewise as in \eqref{charac_torsors}, is an isomorphism of spaces over $X$.
\end{enumerate}
Indeed, we have just discussed the first implication (the fact that $\psi$ is fibrewise bijective follows from the existence of the isomorphisms $\sG|_U \simeq \sP|_U$). And for the converse, we have the following result.

\begin{theorem}\label{sectionImpliesTriviality}
    Let $X$ be a topological space and let $\sG$ be a group space over $X$. Let $\pi : \sP \to X$ be a right $\sG$-space over $X$ such that the map $\psi : \sP \times_X \sG \to \sP \times_X \sP $ defined fibrewise as in \eqref{charac_torsors} is an isomorphism of spaces over $X$. Let $U \subset X$ be an open subspace and assume that $\sP$ admits a section over $U$. Then there exists an isomorphism of right $\sG|_U$-spaces $\sG|_U \simeq \sP|_U$.
\end{theorem}

\begin{proof}
    Let $\sigma : U \to \sP$ be a section of $\sP$ over $U$. Since $\psi$ is an isomorphism of spaces over $X$, an element $p \in \sP|_U$ can be written $p = \sigma(y) \cdot g_y$ for a unique $g_y \in \sG_y$ and the map $\sP|_U \to \sG|_U$ defined by $p \lmt g_{\pi(p)}$ is a $\sG_y$-equivariant bijection. Moreover, it is an isomorphism of $\sG|_U$-spaces because $g_{\pi(p)} = \mathrm{pr}_2 \circ \psi_U^{-1}(\sigma(\pi(p)), p)$ over $U$, where $\mathrm{pr}_2 : \sP \times_X \sG \to \sG$ is the canonical morphism defined by the fibre product of two spaces over $X$.
\end{proof}

Note that Theorem \ref{sectionImpliesTriviality} holds more generally over an arbitrary subspace $Y \subset X$, not necessarily open, with the same proof. We have chosen to restrict the statement to open subspaces of $X$ because this is the only case we shall need (and also in order to avoid confusion later, when we work in the complex analytic category).

\begin{corollary}\label{globalSectionImpliesTriviality}
    Let $\sP$ be a right $\sG$-space such that the map $\psi : \sP \times_X \sG \to \sP \times_X \sP $ defined fibrewise as in \eqref{charac_torsors} is an isomorphism of spaces over $X$. Then $\sP \simeq \sG$ as right $\sG$-spaces if and only if $\sP$ admits a global section over $X$.
\end{corollary}

As a consequence of Definition \ref{torsor_def}, if $\sG$ is locally trivial over $X$, then so is any $\sG$-torsor $\sP$. If $x \in X$ and $U$ is an open neighbourhood of $x$ over which both $\sG$ and $\sP$ are trivial, then $\sG|_U \simeq U \times \sG_x$ as group bundles, and $\sP|_U$ is a principal $\sG_x$-bundle over $U$. So if $X$ is connected and $G$ is a topological group which is the typical fibre of $\sG$ over $X$, then $\sP|_U$ is a principal $G$-bundle over every open set $U$ that trivialises both $\sG$ and $\sP$. However, this does not mean that $\sP$ admits a globally defined $G$-action, so $\sP$ is not necessarily a principal $G$-bundle. In fact, principal $G$-bundles are special cases of $\sG$-torsors, in the following sense.

\begin{exercise}\label{torsors_under_constant_gps}
    Let $G$ be a topological group and let $\sG = G_X := X \times G$ be the constant group with fibre $G$ over $X$. Show that a right $G_X$-action on a space $\sP$ over $X$ is the same as a right $G$-action on $\sP$ by bundle automorphisms. Deduce from this and Definition \ref{torsor_def} that a $G_X$-space $\sP$ is a $G_X$-torsor if and only if it is a principal $G$-bundle.
\end{exercise}

As mentioned before, in these notes we are mostly interested in the case when $X$ is a complex analytic manifold and the group space $\sG$ is analytically locally trivial. Then we can define a $\sG$-torsor either as a complex analytic $\sG$-space $\sP$ (meaning a complex analytic manifold $\sP$ over $X$ endowed with a morphism of complex analytic manifolds $\alpha : \sP \times_X \sG \to \sP$ satisfying the axioms of a fibrewise group action) and such that, locally, there are isomorphisms of right $\sG|_U$-spaces $\sG|_U \simeq \sP|_U$, or as a complex analytic $\sG$-space $\sP \to X$ such that
\begin{enumerate}
    \item the morphism of complex analytic manifolds $q : \sP \to X$ is a \textit{submersion},
    \item the morphism of complex analytic manifolds 
    $$
    \begin{array}{rcl}
        \psi : \sP \times_X \sG & \lra & \sP \times_X \sP \\
        (p, g) & \lmt & (p, p \cdot g)
    \end{array}
    $$
    is an isomorphism.
\end{enumerate}
Indeed, the fact that $q : \sP \to X$ is a submersion guarantees that $\sP$ has local sections and that $\sP \times_X \sP$ is a complex analytic \textit{manifold}.  The fact that $\sP \to X$ is a submersion also guarantees that $\sP \times_X \sG$ is a manifold. As observed before, if $\sG$ is locally trivial and $\sP$ is a $\sG$-torsor in the sense discussed above, then $\sP$ is locally trivial as a holomorphic fibre bundle.

\begin{definition}\label{morphism_of_torsors}
    Let $\sP$ and $\sQ$ be $\sG$-torsors on $X$. A $\sG$-equivariant bundle morphism $u : \sP \to \sQ$ will be called a morphism of $\sG$-torsors.
\end{definition}

With this definition, we see that $\sG$-torsors form a full subcategory of the category of $\sG$-spaces. Moreover, the category of $\sG$-torsors is actually a groupoid.

\begin{proposition}\label{morphisms_of_torsors_are_invertible}
    Let $u : \sP \to \sQ$ be a morphism of $\sG$-torsors over $X$. Then $u$ is an isomorphism.
\end{proposition}

\begin{proof}
    Since $u$ is a bundle morphism, it suffices to show that it is fibrewise invertible. But since $\sP$ and $\sQ$ are locally isomorphic to $\sG$ as $\sG$-spaces, we can assume that $u$ is a bundle morphism from $\sG|_U$ to $\sG|_U$ which is equivariant with respect to the action of $\sG|_U$ on itself by right translations. Such a morphism $u : \sG|_U \to \sG|_U$ satisfies, for all $x \in U$ and all $g \in \sG_x$,
    $$
    u \big( g \big) = u\big( \eps(x) g \big) = u\big( \eps(x) \big) g
    $$
    so it is entirely determined by the image under $u_x$ of the neutral element $\eps(x)$ of the group $\sG_x$. Note that $x = p(g)$ where $p : \sG \to X$, so $u(g) = (u \circ \eps \circ p)(g) g$ on $\sG|_U$ and this is a morphism of spaces over $U$ (the product between $(u\circ\eps\circ p)(g)$ and $g$ takes place in the group $\sG|_{p(g)}$). Moreover, an element $h \in \sG_x$ can be written uniquely in the form $h = u(\eps (x) ) g_h$, namely by taking $g_h := u(\eps(x))^{-1} h$, where $u(\eps(x))^{-1}$ is the inverse of $u(\eps(x))$ in the group $\sG_x$. This again depends homomorphically on $h \in \sG|_U$, so we can define an inverse to $u$ by setting $u^{-1}(h) := g_h$.
\end{proof}

\begin{exercise}
    Check that, when $\sG$ is a group space, the bundle map $u : \sG \to \sG$ defined by $g \mapsto (u \circ \eps \circ p)(g) g$ is $\sG$-equivariant with respect to right translations, and that an inverse of such a map is necessarily $\sG$-equivariant with respect to right translations.
\end{exercise}

Let $\sG \to X$ be a group space. Since $\sG$-torsors are generalizations of principal $G$-bundles (Exercise \ref{torsors_under_constant_gps}), it is natural to expect that $\sG$-torsors can be constructed by gluing trivial torsors (Example \ref{trivial_torsor}) along open subsets. Indeed, let $\sP$ be a (right) $\sG$-torsor (Definition \ref{torsor_def}) and let us fix an open cover $(U_i)_{i\in I}$ of $X$ as well as isomorphisms of $\sG|_{U_i}$-torsors $\phi_i : \sP|_{U_i} \to \sG|_{U_i}$. For all $i, j$ in $I$, we set $U_{ij} := U_i \cap U_j$. Then the complex analytic map
$$
\phi_i \circ \phi_j^{-1}:\sG|_{U_{ij}}\lra \sG|_{U_{ij}}
$$
is an automorphism of the trivial $\sG|_{U_{ij}}$-torsor. As noted in Proposition \ref{morphisms_of_torsors_are_invertible}, there is a bijection 
\begin{equation}\label{autom_of_torsors}
    \left\{\sG|_U \text{-equivariant automorphisms of } \sG|_U \right\} \overset{\simeq}{\longleftrightarrow} \left\{ \text{sections of } \sG|_U \right\}
\end{equation}
given by sending such an automorphism $u : \sG|_U \to \sG|_U$ to the section $s(x) := u(\eps(x))$. So the datum of the map $\phi_i \circ \phi_j^{-1}$ is equivalent to the datum of a section $g_{ij} \in \sO_X(\sG)(U_i \cap U_j)$, where $\sO_X(\sG)$ is the sheaf of regular sections of $\sG$ (and note that this holds without assuming that $\sG$ is locally trivial). The meaning of \textit{regular} here depends on the type of space that $X$ is. If $X$ is a topological space and $\sG$ is a group object over $X$ in the category of topological spaces, regular means continuous. And if $X$ is a complex analytic manifold and $\sG$ is a group object over $X$ in the category of complex analytic manifolds, regular means holomorphic. It is then clear that the regular map $g_{ij} : U_i \cap U_j \to \sG|_{U_{ij}}$ satisfies the cocycle relations $g_{ii} = \eps|_{U_i}$ (the neutral section of $\sG|_{U_i})$) and 
$$
g_{ij}g_{jk} = g_{ik}
$$ 
on $U_i \cap U_j \cap U_k$ (where the product of sections is defined pointwise by the product in $\sG$), because the $\phi_{ij} := \phi_i \circ \phi_j^{-1}$ satisfy $\phi_{ii} = \mathrm{Id}_{\sG|_{U_i}}$ and 
$$
(\phi_i \circ \phi_j^{-1}) \circ (\phi_j \circ \phi_k^{-1}) = \phi_i \circ \phi_k^{-1}$$
over $U_i \cap U_j \cap U_k$. Moreover, any other choice of local trivialisations $(V_j,\psi_j)_{j\in J}$ of the $\sG$-torsor $\sP$ gives rise to an equivalent cocycle $(h_{ij})_{(i,j)}$. Indeed, refining the cover if necessary, we can assume that $I = J$ and that $U_i = V_i$ for all $i \in I$. But then we can write $h_{ij} = u_i g_{ij} u_j^{-1}$, where $u_i \in \sO_X(\sG)(U_i)$ is the regular section of $\sG$ corresponding to the automorphism $\psi_i \circ \phi_i^{-1}$ of the trivial $\sG|_{U_i}$-torsor. Thus, the $\sG$-torsor $\sP$ determines a \v{C}ech cohomology class in $\check{H}^1(X,\sO_X(\sG))$. More generally, isomorphic $\sG$-torsors $\sP$ and $\sQ$ determine the same cohomology class.

\begin{proposition}\label{isomorphicTorsorsAreCohomologous}
    Let $\sG$ be a group space on $X$ and let $\sP$ and $\sQ$ be $\sG$-torsors. Let $(g_{ij})_{(i,j)}$ and $(h_{ij})_{(i,j)}$ be $\sG$-cocycles representing $\sP$ and $\sQ$. Then, $\sP\simeq \sQ$ as $\sG$-torsors if and only if there exists a family of regular sections $u_i : U_i\lra \sG|_{U_i}$ such that 
    $$
    h_{ij}=u_i g_{ij} u_j^{-1}
    $$
    for all $i,j\in I$. Equivalently, $\sP\simeq \sQ$ if and only if they represent the same cohomology class in $\check{H}^1(X,\sO_X(\sG))$.
\end{proposition}

\begin{proof}
    Since $(U_i)_{i\in I}$ trivializes both $\sP$ and $\sQ$, then there are isomorphisms of $\sG|_{U_i}$-torsors
    $$
    \varphi_i:\sP|_{U_i}\lra \sG|_{U_i}\text{ and }\psi_i:\sQ|_{U_i}\lra \sG|_{U_i}.
    $$
    So if $\sP\overset{F}{\lra}\sQ$ is an isomorphism, then $\psi_i\circ F\circ\varphi_i^{-1}:\sG|_{U_i}\to \sG|_{U_i}$ is an automorphism so, in view of the bijection defined in \eqref{autom_of_torsors}, it corresponds to a regular section $u_i:U_i\to \sG|_{U_i}$. And since over the open set $U_{ij} = U_i \cap U_j$, we have a commutative diagram 
    \begin{center}
        \tikzcdset{arrow style=tikz, diagrams={>=angle 90}}
        \begin{tikzcd} [column sep =large, row sep =large]
            \sG|_{U_{ij}} \arrow[d,"u_j",'] \arrow[to=1-3,bend left,"g_{ij}"]& \sP|_{U_{ij}} \arrow[l,"\varphi_j",']\arrow[r,"\varphi_i"]\arrow[d,"F"]& \sG|_{U_{ij}} \arrow[d,"u_i"]\\
            \sG|_{U_{ij}}\arrow[to=2-3,bend right,"h_{ij}",']& \sQ|_{U_{ij}} \arrow[l,"\psi_j",']\arrow[r,"\psi_i"]& \sG|_{U_{ij}}  
        \end{tikzcd}
    \end{center} 
    the conclusion follows. 
\end{proof}

\begin{corollary}
    Under the hypotheses and notation of the previous proposition, $\sP\simeq \sG$ as $\sG$-torsors if and only if $\sP$ can be represented by a cocycle that is actually a coboundary, i.e. $g_{ij}=u_iu_j^{-1}$ for all $i,j\in I$.
\end{corollary}

Conversely, if $p: \sG \to X$ is a group space over $X$, any \v{C}ech $1$-cocycle $(g_{ij} : U_i \cap U_j \to \sG)_{(i,j)}$ gives rise to a $\sG$-torsor $\sP$, whose total space is
$$
\sP := \left(\bigsqcup_{i\in I}\sG|_{U_i} \right)\bigg /\sim
$$
where $\sim$ is the equivalence relation defined as 
$$
(i,h) \sim (j, h') \iff p(h) = p(h') =: x \in U_i \cap U_j \text{ and } h = g_{ij}(x) h' .
$$
The canonical projections $\sG|_{U_i}\lra U_i$ give a well-defined map $\pi:\sP\lra X$ and,
for all $i\in I$, we have a bijection $\varphi_i:\pi^{-1}(U_i)\lra \sG|_{U_i}$ which we can use to endow $\sP$ with topology (and a complex analytic structure if need be). 
The right action of $\sG|_{U_i}$ on itself induces the well-defined (\emph{global}) $\sG$-action 
$$
\begin{array}{rcl}
    \sP\times_X\sG&\lra & \sP \\
   ([h],g)&\longmapsto & [hg]
\end{array}.
$$ 
By construction, for all $i\in I$, the map $\varphi_i$ is $\sG|_{U_i}$-equivariant with respect to this action so $\sP\lra X$ is a $\sG$-torsor, admitting $(g_{ij})_{(i,j)}$ as a cocycle of transition functions. Moreover, if the cocycle $(g_{ij})_{(i,j)}$ comes from a $\sG$-torsor $\sQ$ in the first place, then by Proposition \ref{isomorphicTorsorsAreCohomologous}, we have an isomorphism of $\sG$-torsors $\sP\simeq \sQ$. We can sum up the previous discussion as follows.

\begin{theorem}\label{isomorphismH1andIsoClassesOfTorsors}
Let $\sG\lra X$ be a group space over $X$ and $\sO_X(\sG)$ the sheaf of regular sections of $\sG$.  Then, the \v{C}ech cohomology group $\check{H}^1(X,\sO_X(\sG))$ is in bijection with the set of isomorphism classes of $\sG$-torsors over $X$. 
\end{theorem}

In particular, when $\sG = G_X$ is a constant group, we recover the set of isomorphism classes of principal $G$-bundles on $X$ (see Exercise \ref{torsors_under_constant_gps}).

\subsection{Torsors trivialised by a cover}

Let $\sG$ be a group space on $X$ and let $\sP$ be right $\sG$-torsor. We call $\sP$ \emph{trivial} as a $\sG$-torsor if there exists an isomorphism of $\sG$-torsors $\sP \simeq \sG$. By Corollary \ref{globalSectionImpliesTriviality}, this is equivalent to the existence of a global section for $\sP$. Let us now fix a Galois cover $q : Y \to X$, with automorphism group $\pi = \Aut_X(Y)$, and study $\sG$-torsors $\sP$ over $X$ such that $q^* \sP \simeq q^*\sG$, i.e.\ such that $\sP$ becomes trivial after pullback to the total space of the cover. In the present subsection, we classify such torsors. We then specialise the result to the case when $\sG$ is a $(\pi,G)$-group bundle in the sense of Definition \ref{pi_G_gp_bundle}, meaning that we classify $\sG$-torsors $\sP$ on $X$ that become \textit{trivial principal $G$-bundles} after pullback to $Y$. Such $\sG$-torsors are sometimes called $(\pi, G)$-bundles (\cite{Seshadri_parab_remarks}).

\smallskip

The way to approach this is similar to what we did for Theorem \ref{ClassifGpBdleTrivByCover}. First we observe that if $\sP$ is a $\sG$-torsor on $X$ and $q : Y \to X$ is a Galois cover with automorphism group $\pi$, then $q^* \sP$ is a $\pi$-equivariant $q^* \sG$-torsor on $Y$. It will be convenient to introduce the notation $\hP := q^* \sP$ and $\hG := q^* \sG$. We leave it as an exercise to check that $\hP$ is a $\hG$-torsor. To say that it has a $\pi$-equivariant structure means that there is a collection of maps $(\sigma_f)_{f\in\pi}$ satisfying the following conditions:

\begin{enumerate}
    \item $\sigma_{\mathrm{id}_{Y}} = \mathrm{id}_{\hP}$ and the following diagram is commutative:
    \begin{equation}\label{action_lift_torsors}
        \begin{tikzcd}
            \hP \arrow[to=1-2,"\sigma_f"] \arrow[to=2-1] & \hP \arrow[to=2-2]\\
            Y \arrow[to=2-2,"f"] & Y
        \end{tikzcd}
    \end{equation} 
    \item For all $f_1, f_2 \in \pi$, 
    \begin{equation}\label{lift_compatible_with_product_torsors}
        \sigma_{f_1 \circ f_2} = \sigma_{f_1} \circ \sigma_{f_2}.
    \end{equation}
    \item For all $f\in \pi$, all $y \in Y$, all $p \in \hP_y$ and all $g \in \hG_y$, 
    \begin{equation}\label{lift_compatible_with_action}
    \sigma_f(p \cdot g) = \sigma_f(p) \cdot \tau_f (g).
    \end{equation} where $\tau_f(g)$ denotes the effect on $g$ of the canonical $\pi$-equivariant structure $\tau$ of the group bundle $\hG$, constructed in \eqref{def_action_lift}.
\end{enumerate}

\noindent So if the $\hG$-torsor $\hP$ is trivial, i.e.\ isomorphic to $\hG$ as a right $\hG$-torsor, we inherit a $\pi$-equivariant structure $\sigma$ on the right $\hG$-torsor $\hG$, and these are the structures that we will end up classifying (Theorem \ref{ClassifTorsorTrivByCover}).

\smallskip

In what follows, we denote by $\Gamma(Y, \hG)$ the group of sections of the group bundle $\hG$. Since the fibres of $\hG$ are groups, the set $\Gamma(Y, \hG)$ is indeed a group with respect to pointwise multiplication
\begin{equation}\label{product_on_sections_of_group_space}
    (\alpha_1 \cdot \alpha_2)(y) := \alpha_1(y) \alpha_2(y).
\end{equation}
Moreover, the group $\pi = \Aut_X(Y)$ acts to the right by group automorphisms on the group 
$\Gamma(Y, \hG)$, via 
\begin{equation}\label{action_pi_on_sections}
    \alpha^{f} := \tau_f^{-1} \circ (\alpha \circ f)
\end{equation}
where again $\tau$ is the canonical $\pi$-equivariant structure of the group bundle $\hG$, constructed in \eqref{def_action_lift}. So there is a notion of crossed morphism 
$\psi : \pi \to \Gamma(Y, \hG)$, which we now introduce formally.
\begin{definition}
    A crossed morphism $\psi : \pi \to \Gamma(Y, \hG)$ with respect to the $\pi$-action defined in Equation \ref{action_pi_on_sections}, is a map satisfying 
    \begin{equation}\label{crossed_morphism_into_group_of_sections}
        \psi_{f_1 \circ f_2} = \psi_{f_1}^{f_2} \cdot \psi_{f_2}
    \end{equation} 
    where we denote by $\psi_f : Y \to \hG$ the image of $f \in \pi = \Aut_X(Y)$ under $\psi$.
\end{definition}
Since the action of $\pi$ on $\Gamma(Y, \hG)$ depends on the $\pi$-equivariant structure $\tau$ of $\hG$, we will denote by $Z^1_\tau(\pi, \Gamma(Y, \hG))$ the set of crossed morphisms $\psi : \pi \to \Gamma(Y, \hG)$. Then, the group $\Gamma(Y, \hG)$ acts on $Z^1_\tau(\pi, \Gamma(Y, \hG))$ via
\begin{equation}\label{action_of_sections_on_crossed_morphisms}
    (\alpha \bullet \psi)_f := \alpha^f \cdot \psi_f \cdot \alpha^{-1}
\end{equation}
and where $\alpha \in \Gamma(Y, \hG)$ and $f \in \pi$. To conclude, note that, since each $\tau_f$ is fibrewise a group automorphism, we have
\begin{equation}\label{compatibility_action_with_product_of_sections}
    (\alpha_1 \cdot \alpha_2)^f = \alpha_1^f \cdot \alpha_2^f
\end{equation}
for all $f \in \pi$ and all $\alpha_1, \alpha_2 \in \Gamma(Y, \hG)$.

\begin{exercise}
    Check that the map $(\alpha \bullet \psi) : \pi \to \Gamma(Y, \hG)$ defined in \eqref{action_of_sections_on_crossed_morphisms} is indeed a crossed morphism. Using the compatibility relation \eqref{compatibility_action_with_product_of_sections} between the action of $\pi$ on $\Gamma(Y, \hG)$ and the group structure of $\Gamma(Y, \hG)$, show moreover that $(\alpha_1 \cdot \alpha_2) \bullet \psi = \alpha_1 \bullet (\alpha_2 \bullet \psi)$, so that we have indeed a left action of the group $\Gamma(Y, \hG)$ on the set of crossed morphisms $Z^1_\tau(\pi, \Gamma(Y, \hG))$.
\end{exercise}

We then have the following classification result for $\sG$-torsors that become trivial after pullback to the total space of a Galois cover.

\begin{theorem}\label{ClassifTorsorTrivByCover}
    Let $\sG$ be a group space on $X$. Assume that $X$ is connected and let $q : Y \to X$ be a Galois cover, with automorphism group $\pi := \Aut_X(Y)$.
    \begin{enumerate}
        \item If $\sP$ is a right $\sG$-torsor such that $\hP:= q^* \sP$ is isomorphic to $\hG := q^* \sG$ as a $\hG$-torsor, then there exists a crossed morphism (in the sense of Equation \eqref{crossed_morphism_into_group_of_sections})
        \begin{equation}\label{crossed_morphism_from_pi_to_sections}
            \psi : \pi \lra \Gamma(Y, \hG)
        \end{equation}
        and an isomorphism of $\sG$-torsors $$\sP \simeq \pi \backslash \hG$$ where $f \in \pi$ acts to the left on $g \in \hG_y$ via the $\pi$-equivariant structure
        \begin{equation}\label{twisted_action_of_pi_on_hG}
            \sigma_f (g) := \tau_f\big( \psi_f(y) g \big)
        \end{equation}
        and $\tau$ is the canonical $\pi$-equivariant structure of the group bundle $\hG$, constructed in \eqref{def_action_lift}.
        \item Moreover, if $\psi$ and $\psi'$ are crossed morphisms from $\pi$ to $\Gamma(Y, \hG)$, then the associated $\sG$-torsors $\sP$ and $\sP'$ are isomorphic over $X$ if and only if $\psi$ and $\psi'$ are conjugate under the left action of $\Gamma(Y, \hG)$ defined in \eqref{action_of_sections_on_crossed_morphisms}.
    \end{enumerate} 
    So, the set of isomorphism classes of right $\sG$-torsors that become trivial after pullback to $Y$ is in bijection with the group cohomology set 
    $$
    H^1_\tau\big(\pi, \Gamma(Y, \hG)\big) := 
    \Gamma(Y, \hG) \ \big\backslash \ Z^1_\tau\big(\pi, \Gamma(Y, \hG)\big)\ .
    $$
\end{theorem}

\begin{proof}
    Let $\sP$ be a $\sG$-torsor on $X = \pi \backslash Y$, whose pullback to $Y$ is a trivial torsor $\hP := q^* \sP$, meaning that $\hP$ is isomorphic to $\hG := q^*\sG$ as a (trivial) $\hG$-torsor. We denote by $(\sigma_f)_{f \in \pi}$ the $\pi$-equivariant structure of $\hP$ in the sense of \eqref{action_lift_torsors}. This $\pi$-equivariant structure is constructed as in \eqref{def_action_lift}. Because of Property \eqref{lift_compatible_with_action}, the fact that $\hP$ is a $\hG$-torsor implies that, for all $f \in \pi$, the map $\sigma_f$ is entirely determined by $\sigma_f(\eps(y))$, where $\eps : Y \to \hG$ is the neutral section of the group space $\hG$. Let us now identify $\hP$ with $\hG$ and set, for all $f \in \pi$ and all $y \in Y$,
    \begin{equation}\label{equiv_struc_as_crossed_morphism}
        \psi_f(y) := \tau_f^{-1} \circ \sigma_f\big( \eps(y) \big)
    \end{equation}
    where $(\tau_f)_{f\in\pi}$ is the canonical $\pi$-equivariant structure of $\hG$. Then $\psi_f \in \Gamma(Y, \hG)$ is a global section of the group bundle $\hG$, from which we can reconstruct $\sigma_f$ by setting $\sigma_f(y) := \tau_f (\psi_f(y))$ and verifying that this map satisfies indeed the axioms of a $\pi$-equivariant structure on the $\hG$-torsor $\hG$. Using the definition of $\psi_f$, as well as Property \eqref{lift_compatible_with_product_torsors} and the fact that each $\tau_f$ is fibrewise a group automorphism on $\hG$, a simple calculation shows that $\psi : \pi \to \Gamma(Y, \hG)$ is a crossed morphism in the sense of \eqref{crossed_morphism_into_group_of_sections}, with respect to the $\pi$-action on $\Gamma(Y, \hG)$ defined in \eqref{action_pi_on_sections}. Since there is an isomorphism of $\sG$-torsors $\sP \simeq \pi \backslash \hP$, where $\pi$ acts on $\hG$ via the $\pi$-equivariant structure $(\sigma_f)_{f\in \pi}$, and since $\hP$ is $\pi$-equivariantly isomorphic to the right $\hG$-torsor $\hG$ endowed with the $\pi$-equivariant structure induced by $\psi$, we  have proven the first part of Theorem \ref{ClassifTorsorTrivByCover}. 
    
    \smallskip

    To prove the second part, observe that if $\psi$ and $\psi'$ are crossed morphisms from $\pi$ to $\Gamma(Y,\hG)$ that induce isomorphic $\sG$-torsors $\sP$ and $\sP'$ over $X$, then by pullback there is an automorphism $\alpha$ of the trivial $\hG$-torsor $\hG$ that conjugates the $\pi$-equivariant structures $\sigma$ and $\sigma'$, associated respectively to $\psi$ and $\psi'$. Because every automorphism of the trivial $\hG$-torsor is entirely determined by the image of the neutral section (see the proof of Proposition \ref{morphisms_of_torsors_are_invertible}), the previous observation means that there is a section $\alpha \in \Gamma(Y,\hG)$ such that, all $f \in \pi$, all $y \in Y$ and all $h \in \hG_y$,
    \begin{equation}\label{conjugate_equiv_struc}
        \sigma'_f (h) = \alpha \big(f(y)\big)\ \sigma_f\big(h\, \alpha^{-1}(y)\big)
    \end{equation}
    in $\hG_{f(y)}$. Therefore, using Relation \eqref{equiv_struc_as_crossed_morphism} between $\psi, \psi'$ and $\sigma, \sigma'$, as well as Property \eqref{lift_compatible_with_action} and the fact that $\tau_f$ is fibrewise a group automorphism, Equation \eqref{conjugate_equiv_struc} translates to
    \begin{eqnarray*}
        \psi'_f(y) & = & \tau_f^{-1}\big(\sigma'_f(\eps(y))\big) \\
        & = & \tau_f^{-1} \big( \, \alpha(f(y))\ \sigma_f(\alpha^{-1}(y))\, \big) \\
        & = & \tau_f^{-1} \big( \, \alpha(f(y))\ \sigma_f(\eps(y))\ \tau_f(\alpha^{-1}(y)) \, \big) \\
        & = & \tau_f^{-1}\big( \alpha(f(y)) \big)\ \tau_f^{-1}\big( \sigma_f(\eps(y)) \big) \ \alpha^{-1}(y) \\
        & = & \big( \alpha^f \cdot \psi_f \cdot \alpha^{-1} \big) (y)
    \end{eqnarray*}
    where $\alpha^f$ is defined as in \eqref{action_pi_on_sections}. This relation means that $\psi'$ and $\psi$ are related via the action of $\Gamma(Y,\hG)$ on $Z^1_{\tau}(\pi, \Gamma(Y,G)$ defined in \eqref{action_of_sections_on_crossed_morphisms}, which finishes the proof of the second part of Theorem \ref{ClassifTorsorTrivByCover}.
\end{proof}

If the group space $\sG$ is itself a $(\pi, G)$-group bundle in the sense of Definition \ref{pi_G_gp_bundle}, meaning that $\hG \simeq Y \times G$ as a group bundle, then we can formulate Theorem \ref{ClassifTorsorTrivByCover} slightly differently. Indeed, in that case there is a group isomorphism $\Gamma(Y, \hG) \simeq \Maps(Y, G)$, where the group structure on $\Maps(Y, G)$ is given by pointwise multiplication: for all $\mu_1, \mu_2 \in \Maps(Y,G)$ and all $y \in Y$, 
$$
(\mu_1 \cdot \mu_2)(y) := \mu_1(y) \mu_2(y).
$$
The group $\Maps(Y, \Aut(G))$ acts on $\Maps(Y, G)$ by group automorphisms: for all $\lambda \in \Maps(Y, \Aut(G))$, all $\mu \in \Maps(Y,G)$ and all $y \in Y$, the action is given by 
$$
(\lambda \ast \mu)(y) := \lambda(y)\big( \mu(y) \big).
$$
Note that this action satisfies a compatibility relation with the $\pi$-action on the groups $\Maps(Y, \Aut(G))$ and $\Maps(Y, G)$ given by pre-composition by a map $f \in \pi = \Aut_X(Y)$. Namely, for all $f \in \pi$, all $\lambda \in \Maps(Y,\Aut(G))$ and all $\mu \in \Maps(Y,G)$, we have 
$$
(\lambda \ast \mu) \circ f = (\lambda \circ f) \ast (\mu \circ f).
$$
Since $\sG$ is a $(\pi,G)$-group bundle, Theorem \ref{ClassifGpBdleTrivByCover} shows that the canonical $\pi$-equi\-variant structure of $\hG$ is entirely determined by a crossed morphism $\phi : \pi \to \Maps(Y, \Aut(G))$. Therefore, the $\pi$-action on $\Gamma(Y,\widehat{\sG})$ introduced in \eqref{action_pi_on_sections} can be recast in the following way here: for all $f \in \pi$ and all $\mu \in \Maps(Y,G)$, 
$$
\mu^f := \phi_f^{-1} \ast (\mu \circ f).
$$
We leave it as an exercise to check that this indeed defines an action of $\pi$ on $\Maps(Y,G)$. Then we have a notion of crossed morphism $\rho : \pi \to \Maps(Y, G)$ with respect to that action, which is a map satisfying 
$$
\rho_{f_1 \circ f_2} = \rho_{f_1}^{f_2} \cdot \rho_{f_2}.
$$
In what follows, we denote by $Z^1_\phi(\pi, \Maps(Y,G))$ the set of such crossed morphisms, which, under the assumption that $\hG \simeq Y \times G$, is in bijection with the set $Z^1_{\tau}(\pi, \Gamma(Y,\hG))$ introduced in \eqref{crossed_morphism_into_group_of_sections}. With this identification, the action of $\Gamma(Y, \hG)$ on $Z^1_{\tau}(\pi, \Gamma(Y,\hG))$ introduced in \eqref{action_of_sections_on_crossed_morphisms} becomes the action of the group $\Maps(Y,G)$ on the set of crossed morphisms $Z^1_\phi(\pi, \Maps(Y,G))$ defined, for all $\mu \in \Maps(Y, G)$, all $\rho \in Z^1_\phi(\pi, \Maps(Y,G))$ and all $f \in \pi$, by 
$$
(\mu \bullet \rho)_f := \mu^f \cdot \rho_f \cdot \mu^{-1}.
$$

\begin{exercise}
    Check that the map $(\mu \bullet \rho) : \pi \to \Maps(Y, G)$ defined above is indeed a crossed morphism and that $(\mu_1 \cdot \mu_2) \bullet \rho = \mu_1 \bullet (\mu_2 \bullet \rho)$, so that we indeed have a left action of the group $\Maps(Y, G)$ on the set $Z^1_\phi(\pi, \Maps(Y, G))$.
\end{exercise}

We then have the following corollary of Theorem \ref{ClassifTorsorTrivByCover}, which characterizes the so-called $(\pi,G)$-bundles (= those bundles over $\pi \backslash Y$ which are quotients $\pi \backslash P$ of $\pi$-equivariant principal $G$-bundles $P$ over $Y$ \textit{such that $P$ becomes trivial as a principal $G$-bundle when we forget its equivariant structure}, \cite{Grothendieck_Weil}).

\begin{corollary}\label{pi_G_ppal_bundles}
    In the setting of Theorem \ref{ClassifTorsorTrivByCover}, assume further that the group space $\sG$ is a $(\pi, G)$-group bundle in the sense of Definition \ref{pi_G_gp_bundle} and choose a crossed morphism $\phi : \pi \to \Maps(Y, \Aut(G))$ as in Theorem \ref{ClassifGpBdleTrivByCover}, so that there exists an isomorphism of group bundles 
    $$
    \sG \simeq \pi \backslash \big(Y \times G\big)
    $$ 
    where $f \in \pi$ acts to the left on $(y, g) \in Y \times G$ via
    \begin{equation}
        \tau_f (y,g) := \big(f(y), \varphi_f(y) (g)\big).
    \end{equation}
    Then the set of isomorphism classes of right $\sG$-torsors that become trivial as principal $G$-bundles after pullback to $Y$ is in bijection with the group cohomology set
    \begin{equation}\label{isoClassesOfpi_G_ppal_bdles}
    H^1_{\phi}\big(\pi, \Maps(Y,G)\big) := \Maps(Y,G) \ \big\backslash \ Z^1_{\phi}\big(\pi, \Maps(Y, G)\big)\ .
    \end{equation}
\end{corollary}

A notable special case of Corollary \ref{pi_G_ppal_bundles} is when $X$ is a complex analytic manifold whose universal cover is a contractible Stein manifold (for instance, when $X$ is a compact connected Riemann surface of genus $g_X \geq 1$, or a ball quotient in higher dimension). Then, as we have seen in Section \ref{GpBdleTrivByCover}, every group bundle on $X$ is a $(\piX,G)$-group bundle for some group $G$, so every $\sG$-torsor $\sP$ is a $(\piX,G)$-bundle. Indeed, the pullback of $\sP$ to $\Xt$ is topologically trivial because $\Xt$ is contractible, hence also holomorphically trivial because $\Xt$ is Stein, by a theorem of Grauert \cite{Grauert_hol_bundles}.

\subsection{Frame bundles}

Let $E$ be a fibre bundle on a complex analytic manifold $X$. This means that $E$ is a complex analytic manifold equipped with a \emph{locally trivial} morphism $p : E \to X$. As in Example \ref{groupBundleExampleAutomorphismFibreBundle}, we denote by $\mathcal{A}ut(E)$ the \textit{bundle} of automorphisms of $E$. Indeed, as $E$ is locally trivial as a space over $X$, so is $\mathcal{A}ut(E)$. It is therefore a group bundle on $X$. We now wish to characterise $\mathcal{A}ut(E)$-torsors in terms of the fibre bundle $E$.

\smallskip 

The basic construction goes as follows. Given a right $\mathcal{A}ut(E)$-torsor $\sP$, the group bundle $\mathcal{A}ut(E)$ acts to the left on the product bundle $\sP \times_X E$ via
$$
u \cdot (p, v) := \big(p \cdot u^{-1}, u(v)\big)
$$
where $u \in \mathcal{A}ut(E)_x = \Aut(E_x)$, $p \in \sP_x$ and $v \in E_x$ for some $x \in X$. As $\mathcal{A}ut(E)$ is locally trivial as a space over $X$, so is the right $\mathcal{A}ut(E)$-torsor $\sP$ (Definition \ref{torsor_def}). The bundle $\mathcal{A}ut(E)$ acts (fibrewise) freely on $\sP$, hence also on the locally trivial space $\sP \times_X E$, the associated quotient space
$$
\sP[E] := \mathcal{A}ut(E) \backslash (\sP \times_X E)
$$
is locally trivial as a space over $X$. In particular, $\sP[E]$ is a complex analytic manifold. Moreover, since $\sP$ is locally isomorphic to $\mathcal{A}ut(E)$ as a space over $X$, the bundle $\sP[E]$ is locally isomorphic to $E$ as a space over $X$. This construction induces a functor
\begin{equation}\label{assoc_bdle_to_Aut_E_torsor}
    \begin{array}{rcl}
        \big\{ \mathcal{A}ut(E)\text{-torsors} \big\} & \lra & \big\{ \text{fibre bundles locally isomorphic to } E \big\} \\
        \sP & \lmt & \sP[E]
    \end{array}
\end{equation}
which will end up being an equivalence of categories. Before we proceed with the proof of this, let us observe that if $E = X \times \C^n$, then $\mathcal{A}ut(E) = X \times \GL(n,\C)$, so the $\mathcal{A}ut(E)$-torsors are precisely the principal $\GL(n,\C)$-bundles on $X$ (Exercise \ref{torsors_under_constant_gps}) and the fibre bundles locally isomorphic to $X \times \C^n$ are precisely the rank $n$ vector bundles on $X$. Based on this remark, the quasi-inverse to the functor \eqref{assoc_bdle_to_Aut_E_torsor} should be the functor sending a fibre bundle $E'$ locally isomorphic to $E$ to its so-called \emph{frame bundle} $\Fr(E')$. By definition, a \emph{frame} of the fibre $E'_x$ (at a point $x \in X$) is an isomorphism $E_x \to E'_x$ and the frame bundle of $E'$ is defined, as a set, by
\begin{equation}\label{frame_bundle_def}
    \Fr(E') := \bigsqcup_{x \in X} \Isom(E_x, E'_x).
\end{equation}
Note that the fibres of $\Fr(E')$ are non-empty \emph{because} of the assumption that $E'$ is locally isomorphic to $E$. As a matter of fact, this assumption also shows that $\Fr(E')$ can be endowed with a structure of complex analytic manifold and equipped with a locally trivial morphism $\Fr(E') \to X$. Since $\mathcal{A}ut(E) = \bigsqcup_{x\in X} \Aut(E_x)$, the group bundle $\mathcal{A}ut(E)$ acts to the right on $\Fr(E')$ via $f\cdot u := f \circ u$, for all $f\in \Fr(E'_x) = \Isom(E_x,E'_x)$ and all $u \in \Aut(E_x)$. Evidently, this action is fibrewise free and transitive, making $\Fr(E')$ an $\mathcal{A}ut(E)$-torsor.

\smallskip

It remains to check that the functors $\sP \mapsto \sP[E]$ and $E' \mapsto \Fr(E')$ are quasi-inverse to each other. Let us first assume that $\sP = \Fr(E')$ and compute $\Fr(E')[E]$: we see that the bundle map $\Fr(E') \times_X E \to E'$ sending a pair $(f,v) \in \Fr(E'_x) \times E_x$ to $f(v) \in E'_x$ induces an isomorphism $\Fr(E')[E] \simeq E'$. Conversely, let us assume that $E' = \sP[E]$ and show that $\Fr(\sP[E]) \simeq \sP$. This is more complicated, as it requires noting, first, that a fibre bundle $E'$ which is locally isomorphic to $E$ admits the group space $\mathcal{A}ut(E)$ as structure group, which means that it can be represented, as in Theorem \ref{isomorphismH1andIsoClassesOfTorsors}, by an $\mathcal{A}ut(E)$-valued $1$-cocycle $(g_{ij} \in \sO_{X}(\mathcal{A}ut(E))(U_i \cap U_j))_{(i,j)}$. But then $\Fr(E')$ can be represented by the same $1$-cocycle, using it to glue $\mathcal{A}ut(E)|_{U_i}$ to $\mathcal{A}ut(E)|_{U_j}$ instead of $E'|_{U_i}$ to $E'|_{U_j}$. Next, let us observe that, likewise, if $E' = \sP[E]$, then the fibre bundle $E'$ is representable by the same $\mathcal{A}ut(E)$-valued $1$-cocycle as $\sP$. So, taking $E' = \sP[E]$ and applying the first remark, we see that $\Fr(\sP[E])$ is representable by the same $1$-cocycle as $\sP$, so indeed $\Fr(\sP[E]) \simeq \sP$, by Theorem \ref{isomorphismH1andIsoClassesOfTorsors}. We have thus proved the following result.

\begin{theorem}\label{Aut_E_torsors}
Let $E$ be a fibre bundle over a complex analytic manifold $X$. Then there is an equivalence of categories between fibre bundles that are locally isomorphic to $E$ and torsors under the group bundle $\mathcal{A}ut(E)$. More precisely, the functor 
$$
    \begin{array}{rcl}
        \big\{ \text{fibre bundles locally isomorphic to } E \big\} & \lmt & \big\{ \mathcal{A}ut(E)\text{-torsors} \big\} \\
        E' & \lmt & \Fr(E')
    \end{array}
$$
taking such a fibre bundle $E'$ to its frame bundle is an equivalence of categories, admitting the functor $\sP \mapsto \sP[E]$ defined in \eqref{assoc_bdle_to_Aut_E_torsor} as a quasi-inverse. 
\end{theorem}

As a matter of fact, there is a more general version of Theorem \ref{Aut_E_torsors}, characterizing fibre bundles $E'$ locally isomorphic to a fixed fibre bundle $E$ admitting a fixed group bundle $\sG \hookrightarrow \mathcal{A}ut(E)$ as structure group, meaning that $E$ can be represented by a $\sG$-valued $1$-cocycle, where $\sG$ is a group space acting \textit{faithfully} on $E$, in which case $\sG$ can be identified with a subgroup of the group space $\mathcal{A}ut(E)$. The frame bundle of $E'$ is constructed as in \eqref{frame_bundle_def}, with isomorphisms representable by local sections of $\sG$. So this time $\Fr(E')$ is a $\sG$-torsor and it can be represented by the same $\sG$-valued $1$-cocycle as $E'$. Now, given a $\sG$-torsor $\sP$, we can again form the fibre bundle $\sP[E]$, because $\sG$ acts on $E$ by assumption. In full generality, to obtain a $1$-cocycle $h_{ij}$ for $\sP[E]$, it suffices to compose a $\sG$-valued $1$-cocycle $g_{ij}$ of $\sP$ by the group space morphism $\rho : \sG \to \mathcal{A}ut(E)$ defining the action of $\sG$ on $E$. Then $h_{ij} := \rho \circ g_{ij}$ is in fact a $\rho(\sG)$-valued $1$-cocycle. But if the action of $\sG$ on $E$ is faithful, meaning that $\rho$ is injective, we can recover $g_{ij}$ from $h_{ij}$. Hence an equivalence of categories between fibre bundles $E'$ locally isomorphic to $E$ and admitting the group space $\sG$ as structure group and $\sG$-torsors. This gives a nice geometric interpretation of $\sG$-torsors under any given group space admitting a \emph{faithful} representation as structure group \emph{of a fibre bundle} $E$: they can be viewed as fibre bundles locally isomorphic to $E$ and admitting a reduction of structure group to $\sG \subset \mathcal{A}ut(E)$.

\begin{exercise}\label{Ad_E_torsors}
    Let $\sG$ be a group bundle on a complex analytic manifold $X$ and let $\sP$ be a (right) $\sG$-torsor. Show, by generalizing Example \ref{Adjoint_of_ppal_bundle}, that the group bundle $\mathcal{A}ut(\sP)$ is isomorphic to $\Ad(\sP) := \sP[\sG]$, where $\sG$ acts to the left on $\sP \times_X \sG$ via $g \cdot (p, h) := (p\cdot g^{-1}, ghg^{-1})$. Use Theorem \ref{Aut_E_torsors} to conclude that the category of $\Ad(\sP)$-torsors is equivalent to the category of fibre bundles $\sQ \to X$ that are locally isomorphic to $\sP$ as fibre bundles over $X$.
\end{exercise}

As an application of the notion of frame bundle, we will now show that vector bundles with fixed determinant correspond to torsors under a certain group bundle.

\begin{example}
    Let $E$ be a vector bundle of rank $r$ on $X$ and let $\SL(E)$ be the group bundle consisting of automorphisms of $E$ that induce the identity on the line bundle $\det(E)$ (this group bundle was constructed in Example \ref{exampleSubgroupSL}). We claim that the map sending an $\SL(E)$-torsor $\sP$ to the vector bundle $$\sP[E] := \SL(E) \backslash (\sP \times_X E)\,,$$ defined similarly to \eqref{assoc_bdle_to_Aut_E_torsor}, induces an equivalence of categories between $\SL(E)$-torsors and vector bundles of rank $r$ and determinant isomorphic to $\det(E)$. Indeed, since $\SL(E)$ is a subgroup of the group bundle $\GL(E)$, the fibre bundle $\sP[E]$ is a vector bundle. And since $\sP[E]$ is representable by a cocycle of transition maps which are local sections of $\SL(E)$, its determinant is isomorphic to $\det(E)$. More precisely, to the $\SL(E)$-torsor $\sP$ there is associated a pair $(\sP[E], \lambda)$, where $\sP[E]$ is the vector bundle defined above and $\lambda: \det(E) \to \det(\sP[E])$ is an isomorphism of line bundles. Conversely, if we fix a vector bundle $E'$ locally isomorphic to $E$ (in particular, $E'$ is of rank $r$) and an isomorphism of line bundles $\lambda : \det(E) \to \det(E')$, then, similarly to \eqref{frame_bundle_def},  we can construct the frame bundle
    $$\Fr(E') := \bigsqcup_{x \in X} \Isom_\lambda(E_x, E'_x)$$ where $\Isom_\lambda(E_x, E'_x)$ consists of isomorphisms of vector spaces $u : E_x \to E'_x$ such that $\Lambda^r u = \lambda_x$. As before, the frame bundle $\Fr(E')$ is an $\SL(E)$-torsor and we have $\Fr(\sP[E]) \simeq \sP$ and $\sP(\Fr(E')) \simeq E'$. This concludes the example and we observe that if we want to see vector bundles of rank $r$ and determinant $L$ (a fixed line bundle on $X$) as $\SL(E)$-torsors, it suffices to construct a vector bundle $E$ of rank $r$ and determinant $L$, for instance $E := \C_X^{r-1} \oplus L$. When $L \simeq \C_X$ is a trivial line bundle, we have $\SL(E) \simeq \SL(\C_X^r) \simeq X \times \SL(r,\C)$ and we find again the well-known correspondence between vector bundles of rank $r$ with trivial determinant on the one hand and $\SL(r,\C)$-principal bundles on the other hand.
\end{example}

\section{Flat torsors}\label{twistedLocSyst}

\subsection{Group coverings}

In this subsection we study group bundles with discrete fibres and give a classification result of these in the spirit of Theorem \ref{ClassifGpBdleTrivByCover}.

\begin{definition}
    Let $X$ be a topological space. A group bundle with discrete fibres on $X$ will be called a \emph{group covering}.
\end{definition}

Now we will show that, if the topological space $X$ admits a universal covering, then every group covering $\sG$ on $X$ is a $(\piX, \Gamma)$-bundle for some \emph{discrete} group $\Gamma$. The main difference with Theorem \ref{ClassifGpBdleTrivByCover} is that the crossed morphism $\phi : \pi \to \Maps(Y,\Aut(\Gamma))$ is replaced by a group morphism $\phi : \piX \to \Aut(\Gamma)$ (see Remark \ref{grpBundlesWithDiscreteFibres}).

\begin{theorem}\label{ClassifGpCovering}
    Let $X$ be a connected topological space admitting a universal covering space $\Xt$ and let $p : \zeta \to X$ be a group bundle with discrete fibres. Then there exists a discrete group $\Gamma$, a group morphism $$\varphi : \piX \lra \Aut(\Gamma)$$ and an isomorphism of group coverings $$\zeta \simeq \piX \backslash \big(\Xt \times \Gamma\big)$$ where $f \in \piX$ acts to the left on $(\xi, \gamma) \in \Xt \times \Gamma$ via
    \[
        \tau_f \cdot (\xi,\gamma) := \big(f(\xi), \varphi_f (\gamma)\big).
    \]
    Moreover, if $\phi$ and $\phi'$ are group morphisms from $\piX$ to $\Aut(\Gamma)$, then the associated group coverings $\sG$ and $\sG'$ are isomorphic over $X$ if and only if $\phi$ and $\phi'$ are conjugate under the action of $\Aut(\Gamma)$ on $\Hom(\piX, \Aut(\Gamma))$ defined, for all $\lambda \in \Aut(\Gamma)$ and all $f \in \piX$, by
    $$
    (\lambda \bullet \phi)_f := \lambda \circ \phi_f \circ \lambda^{-1}
    $$

So, the set of isomorphism classes of $(\piX,\Gamma)$-group coverings of $X$ is in bijection with the group cohomology set 
$$
H^1\big(\piX, \Aut(\Gamma) \big) := 
\Aut(\Gamma) \ \big\backslash \ \Hom \big(\piX, \Aut(\Gamma) \big)\ .
$$
\end{theorem}

\begin{proof}
    The proof runs parallel to that of Theorem \ref{ClassifGpBdleTrivByCover}. Let $q : \Xt \to X$ be the universal covering of $X$ and let $q^*\zeta$ be the pullback of $\zeta$ to $\Xt$. Since $\zeta$ is locally trivial with discrete fibres, so is $q^* \zeta$. And since $\Xt$ is simply connected, the group covering $q^* \zeta$ is trivial, meaning that there exists a discrete group $\Gamma$ and an isomorphism $q^* \zeta \simeq \Xt \times \Gamma$. So the only remaining task is to classify all $\piX$-equivariant structures on the product group covering $\Xt \times \Gamma$. As in \eqref{equiv_structure_on_product_bundle}, such an equivariant structure is given by continuous maps $\Phi_f : \Xt \times \Gamma \to \Gamma$, where $f \in \piX$, satisfying $\Phi_f (\xi,\gamma_1 \gamma_2) = \Phi_f (\xi,\gamma_1) \Phi_f (\xi,\gamma_2)$ for all $\xi \in \Xt$ and all $\gamma_1, \gamma_2$ in $\Gamma$. Since each $\Phi_f : \Xt \times \Gamma \to \Gamma$ is continuous and $\Gamma$ is discrete, the fact that $\Xt$ is connected implies that $\Phi_f$ is in fact independent of $\xi$. Hence a map $\phi_f \in \Aut(\Gamma)$. Since these maps come from a $\piX$-equivariant structure on $\Xt \times G$, they satisfy $\phi_{f_1 \circ f_2} = \phi_{f_1} \circ \phi_{f_2}$, meaning that $\phi : \piX \to \Aut(\Gamma)$ is indeed a group morphism, and this concludes the proof of the first part of the theorem. The proof of the second part is similar to that in Theorem \ref{ClassifGpBdleTrivByCover}, with $\lambda \in \Maps(\Xt, \Aut(\Gamma))$ being replaced by $\lambda \in \Aut(\Gamma)$ because a group bundle automorphism $\alpha : \Xt \times \Gamma \to \Xt \times \Gamma$ is of the form $(y , \gamma) \mapsto (y, \Lambda(y, \gamma))$ with $\Lambda : \Xt \times \Gamma \to \Gamma$ continuous. Since $\Gamma$ is discrete and $\Xt$ is connected, the map $\Lambda$ is in fact independent of $\xi \in \Xt$, which induces a group automorphism $\lambda : \Gamma \to \Gamma$.
\end{proof}

Note that, as in Theorem \ref{ClassifGpBdleTrivByCover}, we can classify similarly group coverings that are trivialised by a fixed Galois cover $q : Y \to X$ with automorphism group $\pi$ (it does not have to be the universal cover). Namely, given a discrete group $\Gamma$, the $(\pi,\Gamma)$-group coverings of $X$ are classified by $H^1(\pi, \Aut(\Gamma))$, with $\pi$ acting trivially on $\Aut(\Gamma)$.

\subsection{Torsors under group coverings}

Let us classify torsors under a group covering $\zeta$ of $X$. This is similar to Corollary \ref{pi_G_ppal_bundles}, where we classified torsors under a $(\pi, G)$-group bundle, except that we are now assuming that $\pi = \piX$ and that $G$ is a discrete group $\Gamma$. The main difference with Corollary \ref{pi_G_ppal_bundles} is that the crossed morphism $\rho \in Z^1_{\phi}(\pi, \Maps(Y, G))$ in Corollary \ref{pi_G_ppal_bundles} is now replaced by a crossed morphism $\rho \in Z^1_{\phi}(\piX, G)$. Because of what we saw in Theorem \ref{ClassifGpCovering}, namely that the map $\phi : \piX \to \Aut(\Gamma)$ is a group morphism, saying that $\rho : \piX \to \Gamma$ is a crossed morphism means that, for all $f_1, f_2 \in \piX$, it satisfies
\begin{equation}\label{crossed_morphism_expl}
    \rho\big(f_1 \circ f_2\big) = \rho\big(f_1\big)\ \phi_{f_1}\big(\rho(f_2)\big)    
\end{equation}
where we denote by $\phi_f \in \Aut(G)$ the group automorphism of $G$ associated to $f \in \pi$. We then obtain the following classification result for torsors under a group covering $\zeta$.

\begin{theorem}\label{pi_Gamma_ppal_coverings}
    Let $X$ be a connected topological space admitting a universal covering space $\Xt$ and let $\zeta$ be a group covering on $X$. Denote by $\Gamma$ the typical fibre of $\zeta$ and choose a group morphism $\phi : \piX \to \Aut(\Gamma)$ which represents $\zeta$ in the sense of Theorem \ref{ClassifGpCovering}, i.e.\ such that
    $$
    \zeta \simeq \piX\, \big\backslash \, (\Xt \times \Gamma).
    $$
    Let $\eta$ be a $\zeta$-torsor. Then there exists a crossed morphism $\rho : \piX \to \Gamma$ in the sense of Equation \eqref{crossed_morphism_expl} and an isomorphism of $\zeta$-torsors
    $$
    \eta \simeq \piX \, \big\backslash \, \big(\Xt \times \Gamma\big)
    $$
    where $f \in \piX$ acts to the left on $(\xi, \gamma) \in \Xt \times \Gamma$ via
    \begin{equation}\label{action_via_crossed_morphism}
        \sigma_f (\xi, \gamma) := \big( f(\xi), \rho(f) \phi_f(\gamma) \big)
    \end{equation}    
    Moreover, if $\rho$ and $\rho'$ are crossed morphisms from $\piX$ to $\Gamma$, then the associated $\sG$-torsors $\eta$ and $\eta'$ are isomorphic over $X$ if and only if $\rho$ and $\rho'$ are conjugate under the action of $\Gamma$ on $Z^1_\phi(\piX, \Gamma)$ defined, for all $\gamma \in \Gamma$ and all $f \in \piX$, by
    \begin{equation}\label{Gamma_action_on_crossed_morphisms}
    (\gamma \bullet \rho)(f) := \gamma \, \rho(f) \, \phi_f (\gamma^{-1})\ .
    \end{equation}
    So the set of isomorphism classes of $\zeta$-torsors on $X$ is in bijection with the group cohomology set 
    \begin{equation}\label{isomClassesOfpi_Gamma_ppal_coverings}
        H^1_\phi\big(\piX, \Gamma \big) := 
    \Gamma \ \big\backslash \ Z^1_\phi \big(\piX, \Gamma \big)\ .
    \end{equation}
\end{theorem}

\begin{proof}
    Denote by $\heta$ the pullback of the $\zeta$-torsor $\eta$ to the universal covering $\Xt$ of $X$. Then $\heta$ is a $\pi$-equivariant $\hzeta$-torsor on $\Xt$, where $\hzeta$ is the pullback of $\zeta$ to $\Xt$. Since $\widehat{\zeta} \simeq \Gamma_{\Xt} := \Xt \times \Gamma$ as a group bundle, with $\pi$-equivariant structure given, for all $f \in \pi$, all $\xi \in \Xt$ and all $\gamma \in \Gamma$, by 
    $$
    \tau_f\big(\xi, \gamma\big) = \big( f(\xi), \phi_f(\gamma) ),
    $$
    the $\hzeta$-torsor $\heta$ is actually a principal $\Gamma$-covering on $\Xt$. And since $\Xt$ is simply connected, we have an isomorphism of $\Gamma$-principal coverings $\heta \simeq \Xt \times \Gamma$. Therefore, to prove the first part of the theorem, it suffices to show that the $\pi$-equivariant structure of $\Xt \times \Gamma$ induced by the $\pi$-equivariant structure of $\heta$ is of the form \eqref{action_via_crossed_morphism} for some crossed morphism $\rho : \piX \to \Gamma$. As in the proof of Corollary \ref{pi_G_ppal_bundles}, the compatibility relation $\sigma_f(\xi, \gamma) = \sigma_f(\xi, 1_\Gamma)\,\tau_f(\xi,\gamma)$, between the $\pi$-equivariant structures of $\Xt \times \Gamma$ and $\Gamma_{\Xt}$, implies that $\sigma_f$ is entirely determined by $\sigma_f(\xi,1_\Gamma) \in \Gamma$. As $\Gamma$ is discrete, the function $\xi \mapsto \sigma_f(\xi, 1_\Gamma)$ is constant equal to $\rho(f)\in \Gamma$, say, and this defines a function $\rho : \piX \to \Gamma$ such that, for all $f\in \piX$, all $\xi \in  \Xt$ and all $\gamma \in \Gamma$,
    $$
    \sigma_f \big( \xi, \gamma \big) = \sigma_f \big( \xi, 1_\Gamma \big) \, \tau_f\big(\xi, \gamma \big) = \big( f(\xi), \rho(f) \big) \, \big( f(\xi), \phi_f(\gamma) \big) = \big( f(\xi), \rho(f)\phi_f(\gamma) \big),
    $$
    which proves \eqref{action_via_crossed_morphism}. There remains to check that the map $\rho : \piX \to \Gamma$ thus defined is actually a crossed morphism. But a simple calculation shows that Relation \eqref{crossed_morphism_expl} follows from the definition of $\rho(f)$ as $\sigma_f(\xi, 1_\Gamma)$ and the fact that $\sigma_{f_1 \circ f_2} = \sigma_{f_1} \circ \sigma_{f_2}$, which finishes the proof of the first part of the theorem, since $\eta \simeq \piX \backslash \heta$ by definition of $\heta$. The proof of the second part is similar to what we did in Theorem \ref{ClassifTorsorTrivByCover} and Corollary \ref{pi_G_ppal_bundles}.
\end{proof}

Note that, given a normal subgroup $\pi \lhd \piX$ and a $(\pi,\Gamma)$-group covering $\sG$ on $X$, we can classify similarly those $\sG$-torsors on $X$ \emph{that become trivial as $\Gamma$-principal bundles after pullback to the Galois cover $Y$ of $X$ corresponding to $\pi$}. Namely, given a discrete group $\Gamma$ and a group morphism $\phi : \pi \to \Aut(\Gamma)$, isomorphism classes of torsors under the group covering $\pi \backslash \Gamma_Y$ that they become trivial as $\Gamma$-principal bundles after pullback to $Y$, are in bijection with the elements of the group cohomology set $H^1_\phi(\pi,\Gamma)$.

\subsection{Flat group bundles and flat torsors}\label{twisted_local_systems}

A special case of Theorem \ref{pi_Gamma_ppal_coverings} that is of particular interest to us is when the typical fibre of the group covering $\zeta$ is a discrete group of the form $\Gamma = G^\#$, meaning that $\Gamma$  is the discrete form of a Lie group $G$ (i.e.\ the abstract group $G$ endowed with the discrete topology, as defined in \cite{Atiyah_complex_connections}). In that case, we can view continuous transition functions for the group \emph{covering} $\zeta \simeq \piX \backslash G^\#_{\Xt}$ as \emph{locally constant} transition functions for the group bundle $\sG := \piX \backslash G_{\Xt}$. This motivates the following definition, which we formulate in the complex analytic setting, for this is the case of most interest to us.

\begin{definition}
    Let $X$ be a complex analytic manifold and let $\sG$ be a group bundle on $X$ in the sense of Definition \ref{defGrpBdle}. Assume that $X$ is connected and denote by $G$ the typical fibre of $\sG$. A \emph{flat} group bundle on $X$ is a pair $(\sG, \nabla_\sG)$ consisting of a group bundle $\sG$ and an equivalence class $\nabla_\sG$ of \emph{locally constant} functions
    $$
    \kappa_{ij} : U_i \cap U_j \to \Aut(G)
    $$
    defining a cocyle that represents $\sG$, where two such cocycles are called equivalent if they are related by $\Aut(G)$-valued locally constant functions on the open sets $U_i$ and $U_j$. Such an equivalence class $\nabla_{\sG}$ is called a \emph{flat structure} on $\sG$.
\end{definition}

If $(\sG, \nabla_\sG)$ is a flat group bundle, we can associate to it a \emph{discrete form} $\sG^\#$ of $\sG$, which by definition is the group covering obtained by gluing the various $U_i \times G^\#$ using the locally constant functions $\phi_{ij}$. The group space $\sG^\#$ is then a group bundle with discrete fibre $G^\#$. By Theorem \ref{pi_Gamma_ppal_coverings}, it is therefore a global quotient of the form $\piX \backslash G^\#_{\Xt}$, where the $\piX$-equivariant structure of $G^\#_{\Xt}$ is defined by a group morphism $\phi : \piX \to \Aut(G^\#)$. When this group morphism $\phi$ actually takes value in the group $\Aut(G) \subset \Aut(G^\#)$ consisting of \emph{complex analytic} automorphisms of $G$, then the group bundle $\sG$ is isomorphic, as a \emph{holomorphic} group bundle, to the global quotient $\piX \backslash G_{\Xt}$, where $\piX$ acts on the product group bundle $G_{\Xt}$ via $\phi$. Note that choosing such a presentation of $\sG$ as a quotient is equivalent to the choice of a flat structure on $\sG$, and that the map that forgets the flat structure on a flat group bundle is not injective in general, as it may very well happen that two locally constant $\Aut(G)$-valued cocycles are not equivalent modulo locally constant $\Aut(G)$-valued functions but are equivalent under analytic $\Aut(G)$-valued functions. If we restrict to complex analytic group bundles $\sG$ that become analytically trivial after pullback to the universal cover (which again is all of them if $\Xt$ is a contractible Stein manifold), this is expressed by the fact that the map
$$
H^1_{\phi}\big(\piX, \Aut(G)\big) \to H^1_{\phi} \big( \piX, \Maps(\Xt, \Aut(G)) \big)
$$
is not injective in general. 

\smallskip 

Let us now study torsors under flat group bundles. To that end, recall that flat fibre bundles are defined as (locally trivial) fibre bundles equipped with a flat structure (i.e.\ an equivalence class of locally constant transition functions) and that an isomorphism of flat bundles with fibre $F$, say, is an isomorphism of fibre bundles which can be represented by locally constant functions $u_i : U_i \to \Aut(F)$, a notion that makes sense if we have a fixed equivalence class of locally constant transition functions over double intersections $U_i \cap U_j$. It is convenient to call such an isomorphism a \emph{flat} isomorphism.

\begin{definition}
    Let $X$ be a complex analytic manifold and let $(\sG,\nabla_\sG)$ be a flat group bundle on $X$. A \emph{flat torsor} under $(\sG, \nabla_\sG)$ is a pair $(\sP,\nabla_\sP)$ consisting of a $\sG$-torsor $\sP$ in the sense of Definition \ref{torsor_def} and a flat structure $\nabla_\sP$ on the fibre bundle $\sP$ such that, for all $x \in X$, there is an open neighbourhood $U$ of $x$ and a bundle isomorphism $\sG|_U \simeq \sP|_U$ that is both $\sG|_U$-equivariant and flat.
\end{definition}

The point is that we can classify flat $(\sG,\nabla_\sG)$-torsors. The general way to do so is via an analogue of Theorem \ref{isomorphismH1andIsoClassesOfTorsors}. Recall from the latter that, given a complex analytic manifold $X$ and a group space $\sG$ over $X$, the set of isomorphism classes of $\sG$-torsors is in bijection with the \v{C}ech cohomology group $\check{H}^1(X,\sO_X(\sG))$, where $\sO_X(\sG)$ is the sheaf of holomorphic sections of $\sG$. Now, when a flat structure $\nabla_\sG$ has been fixed on $\sG$, we have a well-defined subsheaf of $\sO_X(\sG)$, consisting of \emph{locally constant sections of $\sG$(with respect to the flat structure $\nabla_\sG$}. We can denote that subsheaf by $\sO_X(\sG^\#)$, since it consists equivalently of continuous (or holomorphic) sections of the analytic group covering $\sG^\#$ (the discrete form of $(\sG,\nabla_\sG)$). We then obtain the following analogue of Theorem \ref{isomorphismH1andIsoClassesOfTorsors}, which is proved along the same lines as that result.

\begin{theorem}\label{isomorphismH1andIsoClassesOfFlatTorsors}
    Let $X$ be a complex analytic manifold and let $(\sG,\nabla_\sG)$ be a flat group bundle on $X$. Then the set of isomorphism classes of flat $(\sG,\nabla_\sG)$-torsors is in bijection with the \v{C}ech cohomology group $\check{H}^1(X,\sO_X(\sG^\#))$, where $\sO_X(\sG^\#)$ is the sheaf of locally constant functions of $\sG$ with respect to the flat structure $\nabla_\sG$.
\end{theorem}

Let us now write the flat bundle $(\sG,\nabla_\sG)$ on $X$ as a quotient $\sG_\phi := \piX \backslash G_{\Xt}$, where $\piX$ acts on the group bundle $G_{\Xt} := \Xt \times G$ via the $\piX$-equivariant structure $\sigma_f(\xi,g) := (f(\xi), \phi_f(g))$ induced by a group morphism $\phi : \piX \to \Aut(G)$. We know from Corollary \ref{pi_G_ppal_bundles} that, if we forget the flat structure of $\sG$, then isomorphism classes of $\sG$-torsors that become trivial after pullback to the universal cover $\Xt$ (which is all of them if $\Xt$ is a contractible Stein manifold) are in bijection with the group cohomology set $H^1_\phi(\piX, \Maps(\Xt, G))$ defined in \eqref{isoClassesOfpi_G_ppal_bdles}. And if we focus on \emph{flat} torsors, then, as in Theorem \ref{pi_Gamma_ppal_coverings}, we can classify isomorphism classes of flat $(\sG,\nabla_\sG)$-torsors using the cohomology set $H^1_\phi(\piX, G)$, defined as in \eqref{isomClassesOfpi_Gamma_ppal_coverings}. For convenience, we formulate this explicitly below.

\begin{definition}\label{def_twisted_local_system}
    Let $X$ be a complex analytic manifold and let $G$ be a complex analytic Lie group. Given a group morphism $\phi : \piX \to \Aut(G)$, we denote by $\sG_\phi := \piX \backslash G_X$ the associated flat group bundle. Then a flat $\sG_\phi$-torsor is called a $\phi$-twisted $G$-local system on $X$. An isomorphism of $\phi$-twisted $G$-local systems on $X$ is, by definition, an isomorphism of flat $\sG_\phi$-torsors.
\end{definition}

Note that the group bundle $\sG_\phi$ comes equipped with a canonical flat structure, so we omit the symbol $\nabla_{\sG_\phi}$ from the notation.

\begin{theorem}\label{classif_twisted_local_systems}
    Let $X$ be a complex analytic manifold and let $G$ be a complex analytic Lie group. Given a group morphism $\phi : \piX \to \Aut(G)$, the map that sends a crossed morphism $\rho : \piX \to G$ to the $G$-twisted local system 
    $$
    \sP(\rho) := \piX \backslash (\Xt \times G)
    $$
    where $\piX$ acts on the principal $G$-bundle $X\times G$ via the $\piX$-equivariant structure 
    $$
    \tau_f\big(\xi, h\big) := \big( f(\xi), \rho(f)\phi_f(h)\big)
    $$
    induces a bijection 
    $$
    H^1_\phi(\piX, G) \simeq \big\{\, \phi\text{-twisted }G\text{-local systems on }X\,\big\}\ \big/\ \text{isomorphism}\, .
    $$
\end{theorem}    

Since an element $\rho(f) \in G$ can be seen as a constant map from $\Xt$ to $G$, we have an induced map
$$
H^1_\phi\big( \piX, G\big) \to H^1_\phi\big( \piX, \Maps(\Xt, G)\big)
$$
given by forgetting the flat structure, which sends the isomorphism class of a flat $\sG_\phi$-torsor to its isomorphism class as a $\sG_\phi$-torsor. It is not injective in general, given that two crossed morphisms $\rho, \rho' : \piX \to G$ that are not conjugate under $G$ may very well be conjugate under $\Maps(\Xt, G)$. We conclude with an example that illustrates the notions we have encountered in these notes.

\begin{example}
    The group bundle $\sG_\phi := \pi \backslash (Y \times \GL(n,\C))$, defined as in Example \ref{GpBdleFromCov} by a degree $2$ Galois cover $q : Y \lra X $ with automorphism group $\pi := \Aut_X(Y) \simeq \Z/2\Z$ and a group morphism $\phi : \pi \to \Aut(\GL(n,\C))$, has a canonical flat structure. By construction, flat $\sG_\phi$-torsors (on $X$) correspond bijectively to $\pi$-equivariant flat principal $\GL(n,\C)$-bundles on $Y$. 
    \begin{itemize}
        \item If the group morphism $\phi$ is trivial, these also correspond, by descent or because $\sG_\phi$ is a constant group, to flat principal $\GL(n,\C)$-bundles on $X$. Equivalently, flat $\sG_\phi$-torsors correspond bijectively to flat rank $n$ vector bundles on $X$. If we denote by $\si : Y \to Y$ the involution given by the Galois action, then $\pi$-equivariant flat vector bundles on $Y$ can also be characterised as pairs $(\mathcal{V},u)$ where $\mathcal{V}$ is a local system of complex vector spaces of rank $n$ on $Y$, and $u : \mathcal{V} \to \sigma^*\mathcal{V}$ is an isomorphism of local systems satisfying the descent condition $\sigma^* u = u^{-1}$ (such pairs $(\mathcal{V},u)$ are called \textit{invariant} local systems on $Y$ and, due to the presence of non-trivial automorphisms for $\mathcal{V}$, the descent condition cannot be omitted from this description).
        \item If $\phi$ is given by the exterior automorphism $g \mapsto \,^t g^{-1}$, then the flat group bundle $\sG_\phi$ is non-constant and flat $\sG_\phi$-torsors can no longer be seen as flat principal $\GL(n,\C)$-bundles on $X$. However, they can be seen as $\pi$-equivariant flat principal $\GL(n,\C)$-bundles on $Y$. Equivalently, these can be characterised as pairs $(\mathcal{V},u)$ where $\mathcal{V}$ is a local system of rank $n$ complex vector spaces on $Y$, and $u : \mathcal{V} \to (\sigma^*\mathcal{V})^\vee$ is an isomorphism of local systems satisfying the descent condition $\,^t(\sigma^* u) = u^{-1}$ (such pairs $(\mathcal{V},u)$ are called \textit{anti-invariant} local systems on $Y$ and, due to the presence of non-trivial automorphisms for $\mathcal{V}$, the descent condition cannot be omitted from this description).
    \end{itemize} 
    In terms of representations of fundamental groups, flat $\sG_\phi$-torsors (which are also called $\phi$-twisted $\GL(n,\C)$-local systems on $X$) correspond to crossed morphisms $\rho: \pi_1 X \to \GL(n,\C)$ for the action of $\pi_1 X$ on $\GL(n,\C)$ induced by the group morphism $\piX \to \piX/\pi_1 Y \simeq \Aut_X(Y)$ and the action of $\Aut_X(Y)$ on $G$ given by $\phi$. Equivalently, such crossed morphisms can be seen as group morphisms $$\widehat{\rho} : \piX \to \GL(n,\C) \rtimes \Z/2\Z$$ such that $\mathrm{pr}_2 \circ \widehat{\rho} = \phi$ (the correspondence being crossed morphisms $\rho$ on the one hand and group morphisms satisfying this condition on the other hand, is given by $\widehat{\rho} := (\rho, \phi)$).
    \begin{itemize}
    \item If $\phi$ is trivial, the semi-direct product above is as direct product, and we get the usual representation space $\mathrm{Hom}(\piX,\GL(n,\C)) / \GL(n,\C)$.
    \item If $\phi$ is not trivial, however, we get a twisted representation space 
    $$\mathrm{Hom}_\phi(\piX, \GL(n,\C) \rtimes \Z/2\Z) / \GL(n,\C)$$ where $\mathrm{Hom}_\phi(\piX, \GL(n,\C) \rtimes \Z/2\Z)$ is the set of group morphisms $\widehat{\rho}$ making the following diagram commute
    \begin{center}
        \tikzcdset{arrow style=tikz, diagrams={>=angle 90}}
        \begin{tikzcd}
            \piX \arrow[to=1-3,"\widehat{\rho}"] \arrow[to=2-2,"\phi",']& & \GL(n,\C) \rtimes \Z/2\Z \arrow[to=2-2,"\mathrm{pr}_2"]\\
            &\Z/2\Z&  
        \end{tikzcd}
    \end{center} 
    \noindent and where $\GL(n,\C)$ acts on such a group morphism $\widehat{\rho}$ via conjugation in the larger group $\GL(n,\C) \rtimes \Z/2\Z)$. These twisted representation spaces consist, by definition, of holonomy representations of $\phi$-twisted $\GL(n,\C)$-local systems on $X$.
    \end{itemize}
\end{example}

% \bibliographystyle{alpha}
% \bibliography{references}

\begin{thebibliography}{Gro58}

    \bibitem[Ati57]{Atiyah_complex_connections}
    M.~F. Atiyah.
    \newblock Complex analytic connections in fibre bundles.
    \newblock {\em Trans. Am. Math. Soc.}, 85:181--207, 1957.
    
    \bibitem[Gra56]{Grauert_hol_bundles}
    H.~Grauert.
    \newblock G\'{e}n\'{e}ralisation d'un th\'{e}or\`eme de {R}unge et application \`a la th\'{e}orie des espaces fibr\'{e}s analytiques.
    \newblock {\em C. R. Acad. Sci. Paris}, 242:603--605, 1956.
    
    \bibitem[Gro55]{Grothendieck_Kansas}
    A.~Grothendieck.
    \newblock A {G}eneral {T}heory of {F}ibre {S}paces with {S}tructure {S}heaf.
    \newblock University of Kansas, Report No. 4, 1955.
    \newblock \url{https://agrothendieck.github.io/divers/kansasscan.pdf}.
    
    \bibitem[Gro58]{Grothendieck_Weil}
    A.~Grothendieck.
    \newblock Sur le m\'emoire de {W}eil. {G}\'en\'eralisation des fonctions ab\'eliennes.
    \newblock In {\em S\'eminaire Bourbaki : ann\'ees 1956/57 - 1957/58, expos\'es 137-168}, number~4 in S\'eminaire Bourbaki, pages 57--71. Soci\'et\'e math\'ematique de France, 1958.
    \newblock talk:141.
    
    \bibitem[Ses10]{Seshadri_parab_remarks}
    C.~S. Seshadri.
    \newblock Remarks on parabolic structures.
    \newblock In {\em Vector bundles and complex geometry. Conference on vector bundles in honor of S. Ramanan on the occasion of his 70th birthday, Madrid, Spain, June 16--20, 2008.}, pages 171--182. Providence, RI: American Mathematical Society (AMS), 2010.
    
    \bibitem[Ste51]{Steenrod}
    N.~Steenrod.
    \newblock {\em {The topology of fibre bundles.}}
    \newblock {Princeton, NJ: Princeton University Press}, 1951.
    
\end{thebibliography}

\end{document}